\documentclass[a4paper,12pt]{amsart}
\usepackage{amsfonts}
\usepackage{amsmath}
\usepackage{amsthm}
\usepackage{amssymb}
\usepackage{mathrsfs}
\usepackage{latexsym}
\usepackage{float}
\usepackage{vmargin}
\setmargrb{3cm}{2.5cm}{2.5cm}{3cm} 
\begin{document}


\newtheorem{theorem}{Theorem}[section]
\newtheorem{definition}[theorem]{Definition}
\newtheorem{lemma}[theorem]{Lemma}
\newtheorem{proposition}[theorem]{Proposition}
\newtheorem{corollary}[theorem]{Corollary}
\newtheorem{example}[theorem]{Example}
\newtheorem{remark}[theorem]{Remark}
\pagestyle{plain}
\title{{P}rojective modules over noncommutative tori are multi-window Gabor frames for modulation spaces}

\author{Franz Luef} 
\thanks{The author was supported by the Marie Curie Excellence grant MEXT-CT-2004-517154 and the Marie Curie Outgoing fellowship PIOF 220464.}
\address{Department of Mathematics\\
University of California, Berkeley\\
CA 94720-3840}
\address{Fakult\"at f\"ur Mathematik\\Universit\"at Wien\\
Nordbergstrasse 15\\ 1090 Wien\\ Austria\\}
\email{franz.luef@univie.ac.at}
\keywords{Gabor frames, noncommutative tori, twisted group $C^*$-algebras, spectral invariance, standard Hilbert $C^*$-module frames}
\subjclass{Primary 42C15, 46L08; Secondary 43A20, 46L89}

\begin{abstract}
In the present investigation we link noncommutative geometry over noncommutative tori with Gabor analysis, where the first has its roots in operator algebras and the second in time-frequency analysis. We are therefore in the position to invoke modern methods of operator algebras, e.g. topological stable rank of Banach algebras, to display the deeper properties of Gabor frames. Furthermore, we are able to extend results due to Connes and Rieffel on projective modules over noncommutative tori to Banach algebras, which arise in a natural manner in Gabor analysis. The main goal of this investigation is twofold: (i) an interpretation of projective modules over noncommutative tori in terms of Gabor analysis and (ii) to show that the Morita-Rieffel equivalence between noncommutative tori is the natural framework for the duality theory of Gabor frames. More concretely, we interpret generators of projective modules over noncommutative tori as the Gabor atoms of multi-window Gabor frames for modulation spaces. Moreover, we show that this implies the {\it existence} of good multi-window Gabor frames for modulation spaces with Gabor atoms in e.g.  Feichtinger's algebra or in Schwartz space. 
\end{abstract}
\maketitle \pagestyle{myheadings} \markboth{F. Luef}{Projective modules over noncommutative tori are multi-window Gabor frames for modulation spaces}
\thispagestyle{empty}


\section{Introduction}
We start with a short review of the first theme of our study: projective modules over $C^*$-algebras and the relevance of Rieffel's work on Morita equivalence of operator algebras.
\par
Rieffel introduced (strong) Morita equivalence for $C^*$-algebras in \cite{ri74-1,ri74-2}, which we call Rieffel--Morita equivalence. The seminal work of Rieffel was motivated by his formulation of Mackey's imprimitivity theorem in terms of $C^*$-algebras. Rieffel--Morita equivalence allows a  classification of $C^*$-algebras which is weaker than a classification up to isomorphisms. The classification of unital $C^*$-algebras with respect to Rieffel--Morita equivalence requires the construction of projective modules over $C^*$-algebras. 
During the 1980's, the research of many operator algebraists concerned projective modules and K-theory for C*-algebras. Another reason for the relevance of projective modules has its origins in Connes' theory of noncommutative geometry \cite{co94-1}. In noncommutative geometry projective modules over noncommutative $C^*$-algebras appear as noncommutative analogue of vector bundles over manifolds, and projective modules over smooth subalgebras of a $C^*$-algebra are viewed as noncommutative analogue of smooth vector bundles over manifolds,\cite{co81-2}. Recall that Connes calls a subalgebra of a $C^*$-algebra smooth if it is stable under the holomorphic function calculus. 
\par
As a demonstration of the power of noncommutative geometry Connes has constructed projective modules over smooth noncommutative tori in \cite{co80}. Rieffel extended Connes' projective modules over noncommutative tori to higher-dimensional noncommutative tori in \cite{ri88}. After these groundbreaking results of Connes and Rieffel, projective modules over noncommutative tori found many applications in mathematics and physics, e.g. Bellissard's interpretation of the integer quantum Hall effect \cite{bescva94}, the work of Marcolli and Mathai on the fractional quantum Hall effect, or the relevance of Rieffel--Morita equivalence of operator algebras in mathematical physics \cite{la99-1}. 
\par
The classification of noncommutative tori up to Rieffel-Morita equivalence relies on the construction of projective modules over noncommutative tori. Rieffel found a general method to construct such in \cite{ri88}. In \cite{lu05-1,lu07} we have shown that Rieffel's construction of projective modules over noncommutative tori \cite{ri88} has a natural formulation in terms of Gabor analysis and we were able to extend his construction to the setting of twisted group algebras. The present work is a continuation of this line of research. We especially want to stress that Connes' theorem \cite{co81-2} on the correspondence between projective modules over a $C^*$-algebra and projective modules over smooth subalgebras of a $C^*$-algebra for noncommutative tori appears naturally in the research about good window classes in Gabor analysis. In joint work with Manin we have shown the relevance of this interpretation for the understanding of quantum theta functions in \cite{luma09}.
\par
Before we are in the position to describe the main theorems of our investigation we want to give a brief exposition of Gabor analysis, the other theme of our investigation. Gabor analysis arose out of Gabor's seminal work in \cite{ga46} on the foundation of information theory. After the groundbreaking work of Daubechies, Grossmann and Meyer, frames for Hilbert spaces have become central objects in signal analysis \cite{dagrme86}, especially wavelets and Gabor frames. In the last years various other classes of frames have been introduced by workers in signal analysis, e.g. curvelets, ridgelets and shearlets. The relevance of Hilbert $C^*$-modules for signal analysis was pointed out out by Packer and Rieffel \cite{pari03,pari04} and Woods in \cite{wo04} for wavelets. 
\par
A Gabor system $\mathcal{G}(g,\Lambda)=\{\pi(\lambda)g:\lambda\in\Lambda\}$ consists of a {\it Gabor atom} $g\in L^2(\mathbb{R}^d)$ and a lattice $\Lambda$ in $\mathbb{R}^d\times\widehat{\mathbb{R}}^d$, where $\pi(\lambda)$ denotes the {\it time-frequency shift} $\pi(\lambda)f(t)=e^{2\pi i\lambda_\omega\cdot t}f(t-\lambda_x)$ for a point $\lambda=(\lambda_x,\lambda_\omega)$ in $\Lambda$. If there exist finite constants $A,B>0$ such that
\begin{equation}
  A\|f\|_2^2\le\sum_{\lambda\in\Lambda}|\langle f,\pi(\lambda)g\rangle|^2\le B\|f\|_2^2
\end{equation}
holds for all $f\in L^2(\mathbb{R}^d)$, then $\mathcal{G}(g,\Lambda)$ is called a {\it Gabor frame} for $L^2(\mathbb{R}^d)$. There is a natural operator associated with a Gabor system $\mathcal{G}(g,\Lambda)$, namely the {\it Gabor frame operator} $S_{g,\Lambda}$ defined as follows:
\begin{equation}
  S_{g,\Lambda}f=\sum_{\lambda\in\Lambda}\langle f,\pi(\lambda)g\rangle\pi(\lambda)g,~~\text{for}~~f\in L^2(\mathbb{R}^d).
\end{equation}
The Gabor frame operator $S_{g,\Lambda}$ is a self-adjoint operator on $L^2(\mathbb{R}^d)$. If $\mathcal{G}(g,\Lambda)$ is a Gabor frame for $L^2(\mathbb{R}^d)$, then an element $f\in L^2(\mathbb{R}^d)$ has a decomposition with respect to the Gabor system $\mathcal{G}(g,\Lambda)$. More precisely, 
\begin{eqnarray*}
  f&=&\sum_{\lambda\in\Lambda}\langle f,\pi(\lambda)(S_{g,\Lambda})^{-1}g\rangle\pi(\lambda)g\\
   &=&\sum_{\lambda\in\Lambda}\langle f,\pi(\lambda)g\rangle\pi(\lambda)(S_{g,\Lambda})^{-1}g\\
   &=&\sum_{\lambda\in\Lambda}\langle f,\pi(\lambda)(S_{g,\Lambda})^{-1/2}g\rangle\pi(\lambda)(S_{g,\Lambda})^{-1/2}g 
\end{eqnarray*}   
for all $f\in L^2(\mathbb{R}^d)$. We call $g_0:=(S_{g,\Lambda})^{-1}g$ the {\it canonical dual} Gabor atom and $\tilde g:=(S_{g,\Lambda})^{-1/2}g$ the {\it canonical tight} Gabor atom of a Gabor frame $\mathcal{G}(g,\Lambda)$. Therefore the invertibility of the Gabor frame operator is essential for the decomposition of a function in terms of Gabor frames. Janssen proved that for Gabor frames $\mathcal{G}(g,\Lambda)$ for $L^2(\mathbb{R}^d)$ with $g\in\mathscr{S}(R^d)$ their canonical dual and tight Gabor atoms $g_0,\tilde{g}$ are again in $\mathscr{S}(\mathbb{R}^d)$. In other words he demonstrated that Gabor frames with good Gabor atoms have dual atoms of the same quality, i.e. all ingredients of the reconstruction formulas are elements of $\mathscr{S}(\mathbb{R}^d)$. The key ingredient in the proof of this deep theorem is the so-called {\it Janssen representation } of the Gabor frame operator \cite{ja95}, which relies on the fact that a Gabor frame operator $S_{g,\Lambda}$ commutes with time-frequency shifts $\pi(\lambda)$ for $\lambda$ in $\Lambda$, i.e. 
 $\pi(\lambda)S_{g,\Lambda}=S_{g,\Lambda} \pi(\lambda)$, for all $\lambda\in\Lambda$.
These commutation relations for Gabor frame operators are the very reason for the rich structure of Gabor systems and the differences between Gabor frames and wavelets, see e.g. \cite{gr01} for further information on this topic. 
\par
The Janssen representation of a Gabor frame operator allows one to express the Gabor frame operator $S_{g,\Lambda}$ in terms of time-frequeny shifts of  the {\it adjoint lattice} $\Lambda^\circ$. The adjoint lattice ${\Lambda^\circ}$ consists of all time-frequency shifts of $\mathbb{R}^{2d}$ that commute with all time-frequency shifts of $\Lambda$, see Section $3$ for an extensive discussion. Now, the Janssen representation of the Gabor frame operator $S_{g,\Lambda}$ of $\mathcal{G}(g,\Lambda)$ with $g\in\mathscr{S}(\mathbb{R}^d)$ is the following 
\begin{equation}
  S_{g,\Lambda}f=\mathrm{vol}(\Lambda)^{-1}\sum_{\lambda^\circ\in\Lambda^\circ}\langle g,\pi(\lambda^\circ)g\rangle\pi(\lambda^\circ)f 
\end{equation}
where $\mathrm{vol}(\Lambda)$ denotes the volume of a fundamental domain of $\Lambda$. The Janssen representation links the original Gabor system $\mathcal{G}(g,\Lambda)$ with a dual system with respect to the adjoint lattice in such a way that the original Gabor frame operator becomes a superposition of time-frequency shifts over the adjoint lattice $\Lambda^\circ$ acting on the function $f$. Therefore Janssen introduced the following Banach algebras \cite{ja95} for $s\ge 0$, where the multiplication is given by twisted convolution, the so-called {\it noncommutative Wiener algebras}:
\begin{equation}
 \mathcal{A}^1_s(\Lambda,c)=\big\{\sum_{\lambda\in\Lambda}a(\lambda)\pi(\lambda):\sum_{\lambda\in\Lambda}|a(\lambda)|(1+|\lambda|^2)^{s/2}<\infty\big\},
\end{equation}
and the smooth noncommutative torus
\begin{equation}
  \mathcal{A}^\infty(\Lambda,c)=\bigcap_{s\ge 0}\mathcal{A}^1_s(\Lambda,c),
\end{equation}
where $c$ refers to the cocycle arising in the composition of time-frequency shifts, see Section $2$ for the explicit expression. Actually Janssen's original approach just worked for lattices $\Lambda=\alpha\mathbb{Z}^d\times\beta\mathbb{Z}^d$ with $\alpha\beta$ a rational number. Gr\"ochenig and Leinert were able to settle the general case in \cite{grle04} by interpreting the result of Janssen as the spectral invariance of $\mathcal{A}^\infty(\Lambda,c)$ in the {\it noncommutative torus} $C^*(\Lambda,c)$, the twisted group $C^*$-algebra of $\Lambda$. Moreover Gr\"ochenig and Leinert were able to show that $\mathcal{A}^1_s(\Lambda,c)$ is a spectral invariant subalgebra of $C^*(\Lambda,c)$. Note that the spectral invariance of a Banach algebra in a $C^*$-algebra implies its stability under the holomorphic function calculus. Therefore $\mathcal{A}^1_s(\Lambda,c)$ and $\mathcal{A}^\infty(\Lambda,c)$ are smooth subalgebras of $C^*(\Lambda,c)$ in the sense of Connes. 
\par
Later we observed in \cite{lu06} that Janssen's result about the spectral invariance of $\mathcal{A}^\infty(\alpha\mathbb{Z}^d\times\beta\mathbb{Z}^d,c)$ in $C^*(\Lambda,c)$ for irrational $\alpha\beta$ had been proved by Connes in his seminal work on noncommutative geometry \cite{co80}. Connes called $\mathcal{A}^\infty(\alpha\mathbb{Z}^d\times\beta\mathbb{Z}^d,c)$ a {\it smooth noncommutative torus} and he considered it the noncommutative analogue of smooth functions on the torus.
\par 
Feichtinger and Gr\"ochenig demonstrated in \cite{fegr89,fegr97} that Gabor frames $\mathcal{G}(g,\Lambda)$ with atoms $g$ in Feichtinger's algebra $M^1(\mathbb{R}^d)$ or in Schwartz's space of test functions $\mathscr{S}(\mathbb{R}^d)$ are Banach frames for the class of modulation spaces. In other words $M^1_s(\mathbb{R}^d)$ and $\mathscr{S}(\mathbb{R}^d)$ are {\it good} classes of Gabor atoms. The crucial tool for these results is the spectral invariance of the noncommutative Wiener algebras and of the smooth noncommutative torus in $C^*(\Lambda,c)$. In a more general setting Gr\"ochenig introduced in \cite{gr04-1,fogr05} the {\it localization theory} for families of Banach spaces, see also \cite{bacahela06} for an approach to localization theory not based on the spectral invariance of Banach algebras.
\par
The good classes of Gabor atoms $M^1_s(\mathbb{R}^d)$ and $\mathscr{S}(\mathbb{R}^d)$ turned out to be the natural building blocks in the construction of projective modules over the noncommutative torus $C^*(\Lambda,c)$. More precisely, in \cite{ri88} Rieffel demonstrated that $\mathscr{S}(\mathbb{R}^d)$ becomes an inner product $\mathcal{A}^\infty(\Lambda,c)$-module for the left action of $\mathcal{A}^\infty(\Lambda,c)$ on $\mathscr{S}(\mathbb{R}^d)$ defined by
\begin{equation*}
  \pi_\Lambda({\bf a})\cdot g=\Big[\sum_{\lambda\in\Lambda}a(\lambda)\pi(\lambda)\Big]g,~~\text{for}~~{\bf a}=\big(a(\lambda)\big)\in\mathscr{S}(\Lambda), g\in\mathscr{S}(\mathbb{R}^d),
\end{equation*}
and the $\mathcal{A}^\infty(\Lambda,c)$-valued inner product
\begin{equation*}
  {_\Lambda}\langle f,g\rangle=\sum_{\lambda\in\Lambda}\langle f,\pi(\lambda)g\rangle\pi(\lambda)~~\text{for}~~f,g\in\mathscr{S}(\mathbb{R}^d).
\end{equation*}
Furthermore, the completion of $\mathscr{S}(\mathbb{R}^d)$ in the norm $_{\Lambda}\|f\|=\| _{\Lambda}\langle f,f\rangle\|_{\mathrm{op}}^{1/2}$ for $f\in\mathscr{S}(\mathbb{R}^d)$ yields a left Hilbert $C^*(\Lambda,c)$-module. In \cite{lu05-1,lu07} we have shown that Rieffel' construction holds for the modulation spaces $M^1_s(\mathbb{R}^d)$ and the noncommutative Wiener algebras $\mathcal{A}^1_s(\Lambda,c)$ for all $s\ge 0$. 
\par
Projective modules over $C^*$-algebra have a natural description in terms of {\it module frames}, which was first noted by Rieffel for finitely generated projective modules and in the general case by Frank and Larson in \cite{frla02}. In \cite{ri88} Rieffel formulated Connes' theorem about projective modules over smooth noncommutative tori in terms of module frames with elements in $\mathscr{S}(\mathbb{R}^d)$. One of our main theorems is the interpretation of Rieffel's result about module frames for projective modules over noncommutative tori as {\it multi-window Gabor frames} for $L^2(\mathbb{R}^d)$ with Gabor atoms in $M^1_s(\mathbb{R}^d)$ and $\mathscr{S}(\mathbb{R}^d)$. Consequently the classification of Rieffel-Morita equivalence for noncommutative tori has as a most important consequence the {\it existence} of multi-window Gabor frames with atoms in $M^1_s(\mathbb{R}^d)$ and $\mathscr{S}(\mathbb{R}^d)$. 
\par
In \cite{gr07} a general class of noncommutative Wiener algebras $\mathcal{A}^1_v(\Lambda,c)$ was studied and the main theorem about $\mathcal{A}^1_v(\Lambda,c)$ is that $\mathcal{A}^1_v(\Lambda,c)$ is spectrally invariant in $C^*(\Lambda,c)$ if and only if $v$ is a GRS-weight, see Section $2$. The main reason for these investigations of Gr\"ochenig was to classify the class of good Gabor atoms. In the present investigation we want to stress that this provides the natural framework for the construction of projective modules over the subalgebras $\mathcal{A}^1_v(\Lambda,c)$ and the generalized smooth noncomutative tori $\mathcal{A}^\infty_v(\Lambda,c)=\bigcap_{s\ge 0}\mathcal{A}^1_{v^s}(\Lambda,c)$ of noncommutative tori for $v$ a GRS-weight. If $v$ is a weight of polynomial growth, we recover the classical theorems of Connes and Rieffel as special case of our main results. 
\par
The paper is organized as follows: in Section 2 we discuss the realization of noncommutative tori as the twisted group $C^*$-algebra $C^*(\Lambda,c)$ of a lattice $\Lambda$ and its subalgebras: the noncommutative Wiener algebras $\mathcal{A}^1_v(\Lambda,c)$ and the generalized smooth noncommutative tori $\mathcal{A}^\infty_v(\Lambda,c)$. These results are strongly influenced by the work of Gr\"ochenig and Leinert on the spectral invariance of noncommutative Wiener algebras $\mathcal{A}^1_v(\Lambda,c)$ in $C^*(\Lambda,c)$ in \cite{grle04,gr07}. We determine the topological stable rank of these subalgebras of $C^*(\Lambda,c)$, which is based on the seminal work of Rieffel in \cite{ri83-1} and the results of Badea on the topological stable rank of spectrally invariant algebras in \cite{ba98-4}. Furthermore, we recall some basic facts about time-frequency analysis and weights on the time-frequency plane. In Section 3 we construct projective modules over noncommutative Wiener algebras $\mathcal{A}^1_v(\Lambda,c)$ and smooth noncommutative tori $\mathcal{A}^\infty_v(\Lambda,c)$, and we use modulation spaces and projective limits of weighted modulation spaces as basic building blocks for the equivalence bimodules over these subalgebras of $C^*(\Lambda,c)$. The main result is classification of $\mathcal{A}^1_v(\Lambda,c)$ and $\mathcal{A}^\infty_v(\Lambda,c)$ up to Rieffel-Morita equivalence. In Section 4 we point out that projective modules over $\mathcal{A}^1_v(\Lambda,c)$ and $\mathcal{A}^\infty_v(\Lambda,c)$ have a natural description in terms of multi-window Gabor frames for $L^2(\mathbb{R}^d)$. Consequently Connes' work about projective modules over smooth subalgebras yields in particular the existence of multi-window Gabor frames with atoms in Feichtinger's algebra or Schwartz space for modulation spaces, which is an interesting consequence of our investigations with great potential for applications in Gabor analysis. Furthermore we invoke a result of Blackadar, Kumjian and Roerdam on the topological stable rank of simple noncommutative tori \cite{blkuro92} to demonstrate that the set of Gabor frames for completely irrational lattices and good windows is dense in $C^*(\Lambda^\circ,c)$.
\section{Noncommutative Wiener algebras and noncommutative tori}
The principal objects of our interest are twisted group algebras for lattices in the time-frequency plane and its enveloping $C^*$-algebras, the twisted group $C^*$-algebras a.k.a. noncommutative tori. Let $\Lambda$ be a lattice in $\mathbb{R}^{2d}$ and $c$ a continuous 2-cocycle with values in $\mathbb{T}$. Then the {\it twisted group algebra} $\ell^1(\Lambda,c)$ is $\ell^1(\Lambda)$ with  
{\it twisted convolution} $\natural$ as multiplication and $^*$ as involution. More precisely, let ${\bf a}=(a(\lambda))_\lambda$ and ${\bf b}=(b(\lambda))_\lambda$ be in $\ell^1(\Lambda)$. Then the {\it twisted convolution} of ${\bf a}$ and ${\bf b}$ is defined by
\begin{equation}
  {\bf a}\natural{\bf b}(\lambda)=\sum_{\mu\in\Lambda}a(\mu)b(\lambda-\mu)c(\mu,\lambda-\mu)~~\text{for}~~\lambda,\mu\in\Lambda,
\end{equation} 
and {\it involution} ${\bf a}^*=\big(a^*(\lambda)\big)$ of ${\bf a}$ given by
\begin{equation}
  a^*(\lambda)=\overline{c(\lambda,\lambda)a(-\lambda)}~~\text{for}~~\lambda\in\Lambda.
\end{equation}
More generally, we want to deal with twisted weighted group algebras $\ell^1_v(\Lambda,c)$ for a suitable weight. A weight $v$ on $\mathbb{R}^{2d}$ is a non-negative measurable function, which satisfies the following properties 
\begin{enumerate}
  \item[(1)] $v$ is {\it submultiplicative}, i.e. $v(x+y,\omega+\eta)\le v(x,\omega)v(y,\eta)$ for all $(x,\omega),(y,\eta)\in\mathbb{R}^{2d}$.
  \item[(2)] $v(x,\omega)\ge 1$ and $v(-x,-\omega)=v(x,\omega)$ for all $(x,\omega)\in\mathbb{R}^{2d}$.
\end{enumerate}
For the rest of the paper we only consider weights $v$ satisfying the conditions $(1)$ and $(2)$, because under these conditions $\ell^1_v(\Lambda)=\{{\bf a}| \sum |a(\lambda)|v(\lambda)=:\|{\bf a}\|_{\ell^1_v}<\infty\}$ has nice properties.
\begin{lemma}\label{ContInvol}
Let $v$ be a weight satisfying the properties $(1)$ and $(2)$. Then $(\ell^1_v(\Lambda),c)$ is a Banach algebra with continuous involution. 
\end{lemma}
\begin{proof}
  Let ${\bf a}$ and ${\bf b}$ in $\ell^1_v(\Lambda)$. Then by the submultiplicativity of $v$ we have that:
  \begin{eqnarray*}
    \|{\bf a}\natural{\bf b}\|_{\ell^1_v}&=&\sum_{\lambda}|\sum_{\mu}a(\mu)b(\lambda-\mu)c(\mu,\lambda-\mu)|v(\lambda)\\
                                         &\le&\sum_{\lambda}\sum_{\mu}|a(\mu)|v(\mu)|b(\lambda-\mu)|v(\lambda-\mu)=\|{\bf a}\|_{\ell^1_v}\|{\bf b}\|_{\ell^1_v}. 
  \end{eqnarray*}
Consequently $\ell^1_v(\Lambda,c)$ is a Banach algebra with respect to twisted convolution. Note that $\ell^1_v(\Lambda,c)$ has a continuous involution if and only if $\|{\bf a}^*\|_{\ell^1_v}\le C\|{\bf a}\|_{\ell^1_v}$ for $C>0$. Since $\|{\bf a}^{**}\|_{\ell^1_v}\le C\|{\bf a}^*\|_{\ell^1_v}\le C^2\|{\bf a}\|_{\ell^1_v}$, and $\|{\bf a}^{**}\|_{\ell^1_v}=\|{\bf a}\|_{\ell^1_v}$, it follows that $C=1$ and $v(-\lambda)=v(\lambda)$. This completes our proof.
\end{proof}
 We refer the interested reader to the survey article \cite{gr07} of Gr\"ochenig for a thorough treatment of weights in time-frequency analysis. 
\par
Now, we want to represent $\ell^1_v(\Lambda,c)$ as superposition of time-frequency shifts on $L^2(\mathbb{R}^d)$. For $(x,\omega)\in\mathbb{R}^{2d}$ we define the {\it time-frequency shift} $\pi(x,\omega)f(t)$ of $f$ by $$\pi(x,\omega)f(t)=M_\omega T_xf(t),$$ where $T_xf(t)=f(t-x)$ is the {\it translation} by $x\in\mathbb{R}^d$ and $M_\omega f(t)=e^{2\pi i t\cdot\omega}f(t)$ is the {\it modulation} by $\omega\in\mathbb{R}^d$. 
\par
Observe, that $(x,\omega)\mapsto\pi(x,\omega)$ is a projective representation of $\mathbb{R}^{2d}$ on $L^2(\mathbb{R}^d)$. This is essentially due to  the following commutation relation between translation and modulation operators:
\begin{equation}\label{TFcom}
  M_\omega T_x= e^{2\pi ix\omega}T_xM_\omega ~~\text{for}~~(x,\omega)\in\mathbb{R}^{2d}.
\end{equation}
The commutation relation \eqref{TFcom} yields a composition law for time-frequency shifts $\pi(x,\omega)$ and $\pi(y,\eta)$:
\begin{eqnarray}
  \pi(x,\omega)\pi(y,\eta)&=&c\big((x,\omega),(y,\eta)\big)\overline{c}\big((y,\eta),(x,\omega)\big) \pi(y,\eta)\pi(x,\omega)\\
  &=&c_{\mathrm{symp}}\big((x,\omega),(y,\eta)\big)\pi(y,\eta)\pi(x,\omega),
\end{eqnarray}
where $c$ denotes the continuous $2$-cocycle $c$ on $\mathbb{R}^{2d}$ defined by $c\big((x,\omega),(y,\eta)\big)=e^{2\pi i y\cdot\omega}$ for $(x,\omega),(y,\eta)\in\mathbb{R}^{2d}$ and $c_{\mathrm{symp}}$ is an {\it anti-symmetric bicharacter} or {\it symplectic bicharacter} on $\mathbb{R}^{2d}$. More explicitly, $c_{\mathrm{symp}}$ is given by 
\begin{equation}
 c_{\mathrm{symp}}\big((x,\omega),(y,\eta)\big)=e^{2\pi i(y\cdot\omega-x\cdot\eta)}=e^{2\pi i\Omega\big((x,\omega),(y,\eta)\big)},
\end{equation} 
where $\Omega\big((x,\omega),(y,\eta)\big)=y\cdot\omega-x\cdot\eta$ is the standard {\it symplectic form} on $\mathbb{R}^{2d}$. 
\par
For a lattice $\Lambda$ in $\mathbb{R}^d\times\widehat{\mathbb{R}}^d$ the mapping of $\lambda\mapsto\pi(\lambda)$ is a projective representation of $\Lambda$ on $L^2(\mathbb{R}^d)$. Now, a projective representations of a lattice $\Lambda$ in $\mathbb{R}^{2d}$ gives a non-degenerate involutive representation of $\ell^1_v(\Lambda,c)$ by
\begin{equation*}
  \pi_\Lambda({\bf a}):=\sum_{\lambda\in\Lambda}a(\lambda)\pi(\lambda)~~\text{for}~~{\bf a}=(a(\lambda))\in\ell^1_v(\Lambda).
\end{equation*}
In other words, $\pi_\Lambda({\bf a}\natural{\bf b})=\pi_\Lambda({\bf a})\pi_\Lambda({\bf b})$ and $\pi_\Lambda(\bf{a}^*)=\pi_\Lambda({\bf a})^*$. Moreover, this involutive representation of $\ell^1_v(\Lambda,c)$ is {\it faithful}, i.e. $\pi_\Lambda({\bf a})=0$ implies ${\bf a}=0$. We refer the reader to \cite{ri88} for a proof of the last assertion.
\par  
We denote the image of the map ${\bf a}\mapsto\pi_\Lambda({\bf a})$ by $\mathcal{A}^1_v(\Lambda,c)$. More explicitly, 
\begin{equation*}
  \mathcal{A}^1_v(\Lambda,c)=\{A\in\mathcal{B}(L^2(\mathbb{R}^d)):A=\sum_{\lambda}a(\lambda)\pi(\lambda),\|{\bf a}\|_{\ell^1_v}<\infty\}
\end{equation*}
is an involutive Banach algebra with respect to the norm $\|A\|_{\mathcal{A}^1_v(\Lambda)}=\sum_{\lambda}|a(\lambda)|v(\lambda)$. We call $\mathcal{A}^1_v(\Lambda,c)$ the {\it noncommutative Wiener algebra} because it is the noncommutative analogue of Wiener's algebra of Fourier series with absolutely convergent Fourier coefficients. 
\par
The involutive Banach algebra $\ell^1_v(\Lambda,c)$ is not a $C^*$-algebra. There exists a canonical construction, which associates to an involutive Banach algebra $\mathcal{A}$ a $C^*$-algebra $C^*(\mathcal{A})$, the {\it universal enveloping $C^*$-algebra} of $\mathcal{A}$. If ${\bf a}\in\ell^1_v(\Lambda,c)$, then one defines a $C^*$-algebra norm $\|{\bf a}\|_{C^*(\Lambda,c)}$ as the supremum over the norms of all involutive representations of $\ell^1_v(\Lambda,c)$ and the {\it twisted group $C^*$-algebra} $C^*(\Lambda,c)$ as the completion of $\ell^1_v(\Lambda,c)$ by $\|.\|_{C^*(\Lambda,c)}$. In the literature ${C^*(\Lambda,c)}$ is also known as {\it noncommutative torus} or {\it quantum torus}. 
If we represent ${C^*(\Lambda,c)}$ as a subalgebra of bounded operators on $L^2(\mathbb{R}^d)$, then $\mathcal{A}^1_v(\Lambda,c)$ is a dense subalgebra of ${C^*(\Lambda,c)}$. 
\par
Now we use the noncommutative Wiener algebras $\mathcal{A}^1_v(\Lambda,c)$ as building blocks for a class of subalgebras  $\mathcal{A}^\infty_v(\Lambda,c)$ of $C^*(\Lambda,c)$ that are noncommutative analogues of smooth functions on a compact manifold. More concretely, we want to deal with smooth noncommutative tori with respect to a general submultiplicative weight. If $v$ is a submultiplicative weight, then we call $\mathcal{A}^{\infty}_v(\Lambda,c)=\bigcap_{s\ge 0}\mathcal{A}^1_{v^s}(\Lambda,c)$ a {\it generalized smooth noncommutative torus}. The subalgebra $\mathcal{A}^{\infty}_v(\Lambda,c)$ of $C^*(\Lambda,c)$ is a complete locally convex algebra whose topology is defined by a family of submultiplicative seminorms $\{\|.\|_{\mathcal{A}^1_{v^s}}|s\ge 0\}$ with 
\begin{equation*}
  \|A\|_{\mathcal{A}^1_{v^s}}=\sum_{\lambda\in\Lambda}|a(\lambda)|v^s(\lambda)~~\text{for}~~A\in\mathcal{A}^\infty_v(\Lambda,c).
\end{equation*} 
In the literature a complete locally convex algebra $\mathcal{A}$ equipped with a family of submultiplicative seminorms is called a {\it locally convex m-algebra} or {\it m-algebra}. It is well-known that m-algebras are precisely the projective limits of Banach algebras. An important class of  m-algebras are Frechet algebras with submultiplicative seminorms. 
\par
By construction several special properties of $\mathcal{A}^{\infty}_v(\Lambda,c)$ are consequences of the structure of $\mathcal{A}^1_{v^s}(\Lambda,c)$, e.g. the spectral invariance in $C^*(\Lambda,c)$. 
\par
Recall that a unital Banach algebra $\mathcal{A}$ is {\it spectrally invariant} in a unital Banach algebra $\mathcal{B}$ with common unit, if for  $A\in\mathcal{A}$ with $A^{-1}\in\mathcal{B}$ implies $A^{-1}\in\mathcal{A}$. The spectral invariance of $\mathcal{A}^1_v(\Lambda,c)$ in $C^*(\Lambda,c)$ was investigated by Gr\"ochenig and Leinert in \cite{grle04}. Their main result shows that this problem only depends on properties of the weight $v$, see \cite{gr07} for the following formulation:
\begin{theorem}[Gr\"ochenig-Leinert]\label{grle}
Let $\Lambda$ be a lattice in $\mathbb{R}^{2d}$. Then the noncommutative Wiener algebra $\mathcal{A}^1_v(\Lambda,c)$ is spectrally invariant in $C^*(\Lambda,c)$ if and only if $v$ is a GRS-weight, i.e. $\lim v(n\lambda)^{1/n}=1$ for all $\lambda\in\Lambda$.
\end{theorem}
For the proof we refer the reader to \cite{gr07}. As a consequence we get the spectral invariance of $\mathcal{A}^\infty(\Lambda,c)$ in $C^*(\Lambda,c)$.
\begin{corollary}\label{GRS-smoothNCT}
Let $\Lambda$ be a lattice in $\mathbb{R}^{2d}$. Then the smooth noncommutative torus $\mathcal{A}^\infty_v(\Lambda,c)$ is spectrally invariant in $C^*(\Lambda,c)$ if and only if $v$ is a GRS-weight.
\end{corollary}
\begin{proof}
Note that we have $\mathcal{A}^1_{v^s}(\Lambda,c)\subset \mathcal{A}^1_{v^{s-1}}(\Lambda,c)\subset\mathcal{A}^1(\Lambda,c)\subset C^*(\Lambda,c)$. Therefore $\mathcal{A}^1_{v^s}(\Lambda,c)$ is spectrally invariant in $C^*(\Lambda,c)$ for all $s$. Consequently,  $\mathcal{A}^\infty_v(\Lambda,c)$  is spectrally invariant in $C^*(\Lambda,c)$. 
\end{proof}
Remark: A submultiplicative weight grows at most exponentially and a GRS-weight has a at most {\it sub-exponential} growth. For an extensive discussion of weights we refer the reader to Chapter $11$ in \cite{gr01} and to \cite{gr07}. 
\par
By the above remark the spectral invariance of $\mathcal{A}^1_v(\Lambda,c)$ and $\mathcal{A}^\infty_v(\Lambda,c)$ in $C^*(\Lambda,c)$ forces $v$ to be sub-exponential. Therefore in the case that $v$ grows faster than a polynomial, the smooth noncommutative torus $\mathcal{A}^\infty_v(\Lambda,c)$ is a subspace of $\mathcal{A}^\infty(\Lambda,c)$. 
\par
An important fact about Gabor frames is that Gabor frames $\mathcal{G}(g,\Lambda)$ with $g$ in the Schwartz space $\mathscr{S}(\mathbb{R}^d)$ provide a discrete description of the Schwartz space $\mathscr{S}(\mathbb{R}^d)$ in terms of its Gabor coefficients. Namely, $f\in \mathscr{S}(\mathbb{R}^d)$ if and only if $(\langle f,\pi(\lambda)g\rangle)\in \mathscr{S}(\Lambda)$. The key to such statements is that the Janssen representation of the Gabor frame operator is in $\mathcal{A}^\infty(\Lambda,c)$. In an analogous manner  the classes $\mathcal{A}^\infty_v(\Lambda,c)$ for $v$ that grows faster than a polynomial provide a description of subspaces $\mathscr{S}_v(\mathbb{R}^d)$ of $\mathscr{S}(\mathbb{R}^d)$ in terms of Gabor frames, see Section 3 for a further discussion of this aspect. The classes $\mathscr{S}_v(\mathbb{R}^d)$ for $v$ that grows faster than a polynomial are natural spaces of test functions for ultra-distributions \cite{cogrro06}. 
\par 
The last theorem has various applications in Gabor analysis, see \cite{grle04,feka04}, most notably that the Gabor frame operator has the same spectrum on all modulation spaces for good Gabor atoms and that the canonical dual and tight Gabor window for good Gabor systems have the same quality as the Gabor atom. These results are based on two observations about spectrally invariant Banach algebras and m-convex algebras $\mathcal{A}$ in $C^*(
\Lambda,c)$: (1) The spectrum $\sigma_{\mathcal{A}}(A)=\sigma_{C^*(\Lambda,c)}(A)$ for $A\in\mathcal{A}$, where  $\sigma_{\mathcal{A}}(A)=\{z\in\mathbb{C}:(z-A)^{-1}~~\text{does not exist in}~~\mathcal{A}\}$ is the spectrum of $A\in\mathcal{A}$. (2) If $\mathcal{A}$ is spectrally invariant in $C^*(\Lambda,c)$, then $\mathcal{A}$ is stable under holomorphic function calculus of $C^*(\Lambda,c)$.
\par
Now, we want to explore the consequences of Gr\"ochenig-Leinert's Theorem \ref{grle} for an understanding of the deeper properties of $\mathcal{A}^1_v(\Lambda,c)$ and $\mathcal{A}^\infty_v(\Lambda,c)$, e.g. their topological stable rank. These results will allow us to draw some important conclusions about the deeper structure of Gabor frames in Section 4.
\par
In \cite{ri83-1} the topological stable rank of a Banach algebra was introduced as a noncommutative analogue of the notion of covering dimension of a 
compact space. In the remaining part of this section we derive some upper bounds for the topological stable rank of the noncommutative Wiener algebras and smooth noncommutative tori.
\par
The {\it left(right) topological stable rank} of a unital topological algebra $\mathcal{A}$, denoted by $\text{ltsr}(\mathcal{A})$  $\big(\text{rtsr}(\mathcal{A})\big)$, is the smallest number $n$ such that the set of $n$-tuples of elements of $A$ which generate $\mathcal{A}$ as a {\it left(right) ideal} is dense in $\mathcal{A}^n$. We denote the set of $n$-tuples of elements of $\mathcal{A}$ which generate $\mathcal{A}^n$ as a left(right) ideal by $\text{Lg}_n(\mathcal{A})$ $\big(\text{Rg}_n(\mathcal{A})\big)$. If ${\text{ltsr}(\mathcal{A})=\text{rtsr}(\mathcal{A})}$, then we call it the {\it topological stable rank} of $\mathcal{A}$, and we denote it by $\text{tsr}(\mathcal{A})$.

\begin{proposition}\label{prop:TopStableRankNonCommWiener}
Let $\Lambda$ be a lattice in $\mathbb{R}^{2d}$ and let $v$ be a GRS-weight. Then 
\begin{equation*}
\mathrm{tsr}(\mathcal{A}^\infty_v\big(\Lambda,c)\big)=\mathrm{tsr}(\mathcal{A}^1_v\big(\Lambda,c)\big)=\mathrm{tsr}\big(C^*(\Lambda,c)\big). 
\end{equation*}
Furthermore, we have that $\mathrm{tsr}(\mathcal{A}^\infty_v\big(\Lambda,c)\big)=\text{tsr}(\mathcal{A}^1_v\big(\Lambda,c)\big)\le 2d+1$.
\end{proposition}
\begin{proof}
Recall that our assumptions on $v$, i.e. $v(-\lambda)=v(\lambda)$ for all $\lambda\in\Lambda$, implies that  
 $\mathcal{A}^1_v(\Lambda,c)$ has a continuous involution. 
 If $\mathcal{A}$ is a unital Banach algebra or m-convex algebra with a continuous involution, then Rieffel proved in \cite{ri83-1} that $\mathrm{ltsr}(\mathcal{A})=\mathrm{rtsr}(\mathcal{A})$. Now, we invoke a result of Badea that $\mathrm{tsr}(\mathcal{A})=\mathrm{tsr}(\mathcal{B})$ if $\mathcal{A}$ is spectrally invariant in $\mathcal{B}$, \cite{ba98-4}. Since $\mathcal{A}^1_v(\Lambda)$ is spectrally invariant in $C^*(\Lambda,c)$ if $v$ satisfies the GRS-condition, \cite{gr07}. Finally, the upper bound for the topological stable rank of $C^*(\Lambda,c)$ is due to Rieffel, see \cite{ri83-1,ri88}. This completes the proof.
\end{proof}
\par
A well-known fact about topological stable rank is that for a topological algebra $\mathcal{A}$ with topological stable rank one the invertible elements are dense in $\mathcal{A}$, e.g. \cite{ri83-1}. By the preceding theorem $\mathrm{tsr}(C^*\big(\Lambda,c)\big)=1$ implies $\mathrm{tsr}(\mathcal{A}^\infty_v\big(\Lambda,c)\big)=\mathrm{tsr}(\mathcal{A}^1_v\big(\Lambda,c)\big)$. It is quite a challenge to determine the topological stable rank of a specific $C^*$-algebra. In the case of noncommutative tori Putnam has shown that the irrational noncommutative 2-torus has topological stable rank one. Later Blackadar, Kumjian and Roerdam extended this result to simple noncommutative tori in \cite{blkuro92}. In our setting this amounts to the following: $C^*(\Lambda,c)$ is simple if and only if $\Lambda$ is {\it completely irrational} in the sense that for every $\lambda\in\Lambda$ there is a $\mu\in\Lambda$ such that $\Omega(\lambda,\mu)$ is an irrational number.
\begin{theorem}
Let $\Lambda$ be completely irrational and $v$ a GRS-weight. Then $\\\mathrm{tsr}(\mathcal{A}^\infty_v\big(\Lambda,c)\big)=\mathrm{tsr}\big(\mathcal{A}^1_v(\Lambda)\big)=1$.
\end{theorem}
\begin{proof}
 The GRS-condition is equivalent to the spectral invariance of $\mathcal{A}^1_v(\Lambda,c)$ in $C^*(\Lambda,c)$ and therefore by Badea's result \cite{ba98-4} to  $\mathrm{tsr}~(\mathcal{A}^\infty_v\big(\Lambda,c)\big)=\mathrm{tsr}~\big(\mathcal{A}^1_v(\Lambda,c)\big)=\mathrm{tsr}~\big(C^*(\Lambda,c)\big)$. By Theorem 1.5 in \cite{blkuro92} we know that $\mathrm{tsr}~C^*(\Lambda,c)=1$ for $\Lambda$ completely irrational. That completes the argument. 
\end{proof}
In Section 3 our main theorems deal with the construction of smooth projective modules over $\mathcal{A}^1_v(\Lambda,c)$ and $\mathcal{A}^\infty_v(\Lambda,c)$ in the sense of Connes, \cite{co81-2}. In other words, we have to construct projections in the algebras $\mathcal{M}_n\big(\mathcal{A}^1_v(\Lambda,c)\big)$ and $\mathcal{M}_n\big(\mathcal{A}^\infty_v(\Lambda,c)\big)$ of $n\times n$ matrices with entries in $\mathcal{A}^1_v(\Lambda,c)$ or $\mathcal{A}^\infty_v(\Lambda,c)$, respectively.
\par
In \cite{le74} Leptin has proved that $\mathcal{M}_n(\mathcal{A})=\mathcal{M}_n\otimes\mathcal{A}$ is spectrally invariant in $\mathcal{M}_n(\mathcal{B})$ if $\mathcal{A}$ is a spectrally invariant Banach subalgebra of $\mathcal{B}$. Note that it is elementary to extend Leptin's result to m-convex algebras. Connes obtained this result for Frechet algebras independently, see \cite{co81-2}, see \cite{sc92} for a discussion of this theorem for m-convex algebras. These observations and the characterization of the spectral invariance of $\mathcal{A}^1_v(\Lambda,c)$ and $\mathcal{A}^\infty_v(\Lambda,c)$ in $C^*(\Lambda,c)$, see Theorem \ref{grle} and our Proposition \ref{prop:TopStableRankNonCommWiener}, yield the following result.  
\begin{theorem}
  Let $\Lambda$ be a lattice in $\mathbb{R}^{2d}$. Then $\mathcal{M}_n\big(\mathcal{A}^1_v(\Lambda,c)\big)$ and $\mathcal{M}_n\big(\mathcal{A}^\infty_v(\Lambda,c)\big)$ are spectrally invariant in $\mathcal{M}_n\big(C^*(\Lambda,c)\big)$ if and only if $v$ is a GRS-weight. 
\end{theorem}
Following Connes we deduce from the previous theorem the density result for K-groups, see \cite{co94-1}.
\begin{corollary}
Let $\Lambda$ be a lattice in $\mathbb{R}^{2d}$ and $v$ a GRS-weight. Then the inclusion $i$ of $\mathcal{A}^1_v(\Lambda,c)$ and $\mathcal{A}^\infty_v(\Lambda,c)$ into $C^*(\Lambda,c)$, respectively, gives an isomorphism of $K_0$-groups $i_0:K_0\big(\mathcal{A}\big)\to K_0\big(C^*(\Lambda,c)\big)$ and of $K_1$-groups $i_1:K_0\big(\mathcal{A}\big)\to K_0\big(C^*(\Lambda,c)\big)$ for $\mathcal{A}=\mathcal{A}^1_v(\Lambda,c)$ and $\mathcal{A}=\mathcal{A}^\infty_v(\Lambda,c)$.
\end{corollary}
If the weight $v$ grows at most polynomially, then the preceding Corollary corresponds to the well-known result for the smooth noncommutative torus $\mathcal{A}^\infty(\Lambda,c)$, \cite{co80}. One of the main goals of this section was to demonstrate that Connes's results on the spectral invariance of the smooth noncommutative torus $\mathcal{A}^\infty(\Lambda,c)$ in $C^*(\Lambda,c)$ is the special case $\mathcal{A}^\infty_{v}(\Lambda,c)$ for the radial weight $v(\lambda)=(1+|\lambda|^2)$ of a general class of subalgebras $\mathcal{A}^\infty_{v}(\Lambda^c)$ of $C^*(\Lambda,c)$ and that these algebras are intimately tied with recent developments in time-frequency analysis.  
\section{Projective modules over noncommutative tori and noncommutative Wiener algebras}

In \cite{lu07} we have shown that Feichtinger's algebra $S_0(\mathbb{R^d})$ provides a convenient class of functions for the construction of Hilbert $C^*(\Lambda,c)$-modules. In the present section we extend the results in \cite{lu07} to modulation spaces $M^1_v(\mathbb{R}^d)$ and to $\mathscr{S}_v(\mathbb{R}^d)=\bigcap_{s\ge 0}M^1_{v^s}(\mathbb{R}^d)$ for a non-trivial submultiplicative weight $v$. 
\par
At this point it will be convenient to formally introduce an important class of function spaces, invented by Feichtinger in \cite{fe83-4}. Modulation spaces have found many applications in harmonic analysis and time-frequency analysis, see the interesting survey article \cite{fe06} for an extensive bibliography. If $g$ is a window function in $L^2(\mathbb{R}^d)$, then the {\it short-time Fourier transform} (STFT) of a function or distribution $f$ is defined by 
\begin{equation}
  V_gf(x,\omega)=\langle f,\pi(x,\omega)g\rangle=\int_{\mathbb{R}^d}f(t)\overline{g}(t-x)e^{-2\pi ix\cdot\omega}dt.
\end{equation}
The STFT of $f$ with respect to the window $g$ measures the time-frequency content of a function $f$. Modulation spaces are classes of function spaces, where the norms are given in terms of integrability or decay conditions of the STFT. In the present section we restrict our interest to Feichtinger's algebra $M^1(\mathbb{R}^d)$ and its weighted versions $M^1_v(\mathbb{R}^d)$ for a submultiplicative weight $v$. We introduce the full class of modulation spaces in Section 4, where we interpret the projective modules over $C^*(\Lambda,c)$ of the present section as multi-window Gabor frames over modulation spaces.
\par
In time-frequency analysis the modulation space $M^1_v(\mathbb{R}^d)$ has turned out to be a good class of windows. If $\varphi(t)=e^{-\pi t^2}$ is the Gaussian, then the modulation space $M^1_v(\mathbb{R}^d)$ is the space 
\begin{equation*}
  M^1_v(\mathbb{R}^d)=\{f\in L^2(\mathbb{R}^d): \|f\|_{M^1_v}:=\int_{\mathbb{R}^d}|V_\varphi f(x,\omega)|v(x,\omega)dxd\omega<\infty\}.
\end{equation*}
The space $M^1(\mathbb{R}^d)$ is the well-known Feichtinger algebra $S_0(\mathbb{R}^d)$, which he introduced in \cite{fe81-3} as the minimal strongly character invariant Segal algebra.
\par 
Let $v$ be a submultiplicative weight on $\mathbb{R}^{2d}$ such that $v$ is not constant on $\mathbb{R}^d$ and $\widehat{\mathbb{R}}^d$. Then a natural generalization of Schwartz's class of test functions is given by
\begin{equation*}
  \mathscr{S}_v(\mathbb{R}^d):=\bigcap_{s\ge 0}M^1_{v^s}(\mathbb{R}^d)
\end{equation*}
with seminorms $\|f\|_{M^1_{v^s}}$ for $s\ge 0$. If $v$ is of at most polynomially growth, then $ \mathscr{S}_v(\mathbb{R}^d)$ is the Schwartz class of test functions $\mathscr{S}(\mathbb{R}^d)$, \cite{gr01}. For a submultiplicative weight $v$ that grows faster than a polynomial the Gelfand-Shilov space $S_{{\scriptstyle\frac{1}{2},\frac{1}{2}}}(\mathbb{R}^d)$ is contained in $\mathscr{S}_v(\mathbb{R}^d)$, see \cite{cogrro06}. In the main results about projective modules over $C^*(\Lambda,c)$ the space $M^1_v(\mathbb{R}^d)$ serves as pre-equivalence bimodule between $\mathcal{A}^1_v(\Lambda,c)$ and $\mathcal{A}^1_v(\Lambda^\circ,\overline{c})$ and in an analogous manner $\mathscr{S}_v(\mathbb{R}^d)$ is shown to be a pre-equivalence bimodule between $\mathcal{A}^\infty_v(\Lambda,c)$ and $\mathcal{A}^\infty_v(\Lambda^\circ,\overline{c})$. 
\par
The spaces $M^{1}_v(\mathbb{R}^d)$ and $\mathscr{S}_v(\mathbb{R}^d)$ have many useful properties, see \cite{fe83-4,gr01}. In the following proposition we collect those facts which we need in the construction of the projective modules over $C^*(\Lambda,c)$. 
\begin{proposition}\label{ModSpaces}
Let $v$ be a non-trivial submultiplicative weight.
  \begin{enumerate}
    \item For $g\in M^{1}_v(\mathbb{R}^d)$ we have $\pi(y,\eta)g\in M^{1}_v(\mathbb{R}^d)$ for $(y,\eta)\in\mathbb{R}^{2d}$ with
\begin{equation*}    
    \|\pi(y,\eta)g\|_{M^{1}_v}\le v(y,\eta)\|g\|_{M^{1}_v}.
\end{equation*}    
    \item If $f,g$ are in $M^1_v(\mathbb{R}^d)$, then $V_gf\in M^1_{v\otimes v}(\mathbb{R}^{2d})$.
    \item Let ${\bf a}=\big(a(\lambda)\big)$ be in $\ell^1_v(\Lambda)$ and $g\in M^1_v(\mathbb{R}^d)$. Then $\sum_{\lambda\in\Lambda}a(\lambda)\pi(\lambda)g$ is in $M^1_v(\mathbb{R}^d)$ with 
    \begin{equation*}
      \Big\|\sum_{\lambda\in\Lambda}a(\lambda)\pi(\lambda)g\Big\|_{M^1_v}\le \|{\bf a}\|_{\ell^1_v}\|g\|_{M^1_v}.
    \end{equation*}
    \item If $f,g$ are in $M^1_v(\mathbb{R}^d)$, then $\big(V_gf(\lambda)\big)\in\ell^1_v(\Lambda)$. 
  \end{enumerate}
\end{proposition}
\begin{corollary}\label{SchwartzLikeSpaces}
Let $v$ be a submultiplicative weight.
  \begin{enumerate}
    \item For $g\in\mathscr{S}_v(\mathbb{R}^d)$ we have $\pi(y,\eta)g\in\mathscr{S}_v(\mathbb{R}^d)$ for $(y,\eta)\in\mathbb{R}^{2d}$ with
 \begin{equation*}    
    \|\pi(y,\eta)g\|_{M^{1}_{v^s}}\le v(y,\eta)\|g\|_{M^{1}_{v^s}}~~ \text{for all}~~s\ge 0.
 \end{equation*}   
    \item If $f,g$ are in $\mathscr{S}_v(\mathbb{R}^d)$, then $V_gf\in \mathscr{S}_{v\otimes v}(\mathbb{R}^{2d})$.
    \item Let ${\bf a}=\big(a(\lambda)\big)$ be in $\mathscr{S}_v(\Lambda)=\bigcap_{s\ge 0}\ell^1_{v^s}(\Lambda)$ and $g\in \mathscr{S}_v(\mathbb{R}^d)$. Then $\sum_{\lambda\in\Lambda}a(\lambda)\pi(\lambda)g$ is in $\mathscr{S}_v(\mathbb{R}^d)$ with 
    \begin{equation*}
      \Big\|\sum_{\lambda\in\Lambda}a(\lambda)\pi(\lambda)g\Big\|_{M^1_{v^s}}\le \|{\bf a}\|_{\ell^1_{v^s}}\|g\|_{M^1_{v^s}},~~ \text{for all}~~s\ge 0.
    \end{equation*} 
    \item If $f,g$ are in $\mathscr{S}_v(\mathbb{R}^d)$, then $\big(V_gf(\lambda)\big)\in\mathscr{S}_{v}(\Lambda)$.
  \end{enumerate}
\end{corollary}
We refer the reader to \cite{gr01} for a proof of these statements about $g\in M^{1}_v(\mathbb{R}^d)$ and $\mathscr{S}_v(\mathbb{R}^d)$.
\par
We continue our presentation with a brief discussion of the Fundamental Identity of Gabor analysis, which is an identity for the product of two STFTs. This identity is the essential tool in Rieffel's construction of projective modules over noncommutative tori in \cite{ri88}. Later, Janssen, Tolimieri, Orr observed independently the relevance of this identity in Gabor analysis, therefore Janssen called it the {\it Fundamental Identity of Gabor analysis} (FIGA). Feichtinger and Kozek generalized these results in \cite{feko98} to Gabor frames with lattices in elementary locally compact abelian groups, because they realized that the Poisson summation formula for the symplectic Fourier transform is the main ingredient in the proof of the FIGA. Actually, in the approach of Feichtinger and Kozek to FIGA they have rediscovered the main arguments of Rieffel's discussion in \cite{ri88}. In \cite{felu06} we have extended the results of Feichtinger, Kozek and Rieffel in a discussion of dual pairs of Gabor windows. In the following we present a slightly more general version of the main theorem in \cite{felu06}. 
\par
We already mentioned that the FIGA follows from an application of the Poisson summation formula for the symplectic Fourier transform. In the symplectic version of the Poisson summation formula the {\it adjoint lattice} of a lattice $\Lambda$ is the substitute of the dual lattice in the Euclidean Poisson summation formula. More precisely, if $\Lambda$ is a lattice in $\mathbb{R}^{2d}$, then in \cite{feko98} Feichtinger and Kozek defined its {\it adjoint lattice} by 
\begin{equation*}
  \Lambda^\circ=\{(x,\omega)\in\mathbb{R}^{2d}: c_{\mathrm{symp}}\big((x,\omega),\lambda\big)=1~~\text{for all}~~\lambda\in\Lambda\}.
\end{equation*}
In \cite{ri88} Rieffel denoted the lattice $\Lambda^\circ$ by $\Lambda^\perp$ and called it the {\it orthogonal lattice}.
\begin{theorem}[FIGA]\label{FIGA}
Let $\Lambda$ be lattice in $\mathbb{R}^{2d}$. Then for $f_1,f_2,g_1,g_2\in M^1_v(\mathbb{R}^d)$ or in $\mathscr{S}_v(\mathbb{R}^d)$ the following identity holds:
\begin{equation}
  \sum_{\lambda\in\Lambda}\langle f_1,\pi(\lambda)g_1\rangle\langle \pi(\lambda)g_2,f_2\rangle={\mathrm{vol}(\Lambda)}^{-1}\sum_{\lambda^\circ\in\Lambda^\circ}\langle f_1,\pi(\lambda^\circ)f_2\rangle\langle\pi(\lambda^\circ)g_2, g_1\rangle,
\end{equation}
where $\mathrm{vol}(\Lambda)$ denotes the volume of a fundamental domain of $\Lambda$.
\end{theorem}
The case $v(x,\omega)=(1+|x|^2+|\omega|^2)$ for $\mathcal{S}_v(\mathbb{R}^d)$ was proved by Rieffel in Proposition $2.11$ in \cite{ri88}. 
\par
In \cite{lu07} we observed that Rieffel' construction holds for $M^1(\mathbb{R}^d)$. In the present investigation we want to emphasize that the method of Rieffel also works for $M^1_v(\mathbb{R}^d)$ and $\mathscr{S}_v(\mathbb{R}^d)$ and provides new classes of pre-Hilbert $C^*(\Lambda,c)$-modules. 
\par
We define a left action of $\mathcal{A}^1_v(\Lambda,c)$ on $M^1_v(\mathbb{R}^d)$ by
\begin{equation*}
  \pi_\Lambda({\bf a})\cdot g=\big[\sum_{\lambda\in\Lambda}a(\lambda)\pi(\lambda)\big]g~~\text{for}~~{\bf a}\in\ell^1_v(\Lambda),g\in M^1_v(\mathbb{R}^d).
\end{equation*}
If $f,g$ are in $M^1_v(\mathbb{R}^d)$, then $\big(V_gf(\lambda)\big)$ is in $\ell^1_v(\Lambda)$. Consequently,
we have that
\begin{equation*}
  {_\Lambda}\langle f,g\rangle=\sum_{\lambda\in\Lambda}\langle f,\pi(\lambda)g\rangle\pi(\lambda)
\end{equation*}   
is an element of $\mathcal{A}^1_v(\Lambda,c)$. The crucial observation is that ${_\Lambda}\langle f,g\rangle$ is a $\mathcal{A}^1_v(\Lambda,c)$-valued  inner product. 
In the following theorem we prove that $M^1_v(\mathbb{R}^d)$ becomes a full left Hilbert $C^*(\Lambda,c)$-module ${_\Lambda}V$ when completed with respect to the norm ${_\Lambda}\|f\|=\|{_\Lambda}\langle f,f\rangle\|_{\mathrm{op}}^{1/2}$ for $f\in M^1_v(\mathbb{R}^d)$.  
\begin{theorem}\label{thm:HilbertModulesLeft}
Let $\Lambda$ be a lattice in $\mathbb{R}^{2d}$. If $v$ is a submultiplicative weight, then $M^1_v(\mathbb{R}^{d})$ is a left pre-inner product $\mathcal{A}^1_v(\Lambda,c)$-module for the left action of $\mathcal{A}^1_v(\Lambda,c)$ on $M^1_v(\mathbb{R}^{d})$
\begin{equation*}
  \pi_\Lambda({\bf a})\cdot g=\sum_{\lambda\in\Lambda}a(\lambda)\pi(\lambda)g~~\text{for}~~{\bf a}=\big(a(\lambda)\big)\in\ell^1_v(\Lambda),g\in M^1_v(\mathbb{R}^{d}),		
\end{equation*}
the $\mathcal{A}^1_v(\Lambda,c)$-inner product
\begin{equation*}
  _{\Lambda}\langle f,g\rangle=\sum_{\lambda\in\Lambda}\langle f,\pi(\lambda)g\rangle\pi(\lambda)~~\text{for}~~f,g\in M^1_v(\mathbb{R}^{d})
\end{equation*}
and the norm $_{\Lambda}\|f\|=\|_{\Lambda}\langle f,f\rangle\|_{\mathrm{op}}^{1/2}$. 
\end{theorem}
\begin{proof}
  We briefly sketch the main steps of the proof, since the discussion follows similar lines as in \cite{ri88} and \cite{lu07}. 
\begin{enumerate}
  \item[(a)] If ${\bf a}\in\ell^1_v(\Lambda)$ and $g\in M^1_v(\mathbb{R}^d)$, then ${\bf a}\mapsto\pi_\Lambda({\bf a})\cdot g$
  is in $M^1_v(\mathbb{R}^d)$, see Prop. \ref{ModSpaces}. Therefore, the left action $\pi_\Lambda(a)\cdot g$ is a well-defined and bounded map on $M^1_v(\mathbb{R}^d)$.
  \item[(b)] The compatibility of the left action with the $\mathcal{A}^1_v(\Lambda,c)$-inner product amounts to the following identity 
  \begin{equation*}\label{comp} 
   _{\Lambda}\langle \pi_\Lambda({\bf a})\cdot f,g\rangle=\pi_\Lambda({\bf a}){_\Lambda}\langle f,g\rangle
  \end{equation*}  
  for all $f,g\in M^1_v(\mathbb{R}^d)$ and ${\bf a}\in\ell^1_v(\Lambda)$, which follows from the following computation:
 \begin{eqnarray*}
    {_\Lambda}\langle \pi_\Lambda({\bf a})\cdot f,g\rangle&=&\sum_{\mu\in\Lambda}\langle\pi_\Lambda({\bf a})\cdot f,\pi(\mu)g\rangle\pi(\mu)\\
                              &=&\sum_{\mu\in\Lambda}\sum_{\lambda\in\Lambda}a(\lambda)\langle\pi(\lambda)f,\pi(\mu)g\rangle\pi(\mu)\\
                              &=&\sum_{\lambda,\mu}a(\lambda)\langle f,\pi(\lambda)^*\pi(\mu)g\rangle\pi(\mu)\\
                              &=&\sum_{\lambda,\mu}a(\lambda)\langle f,\pi(\mu-\lambda)g\rangle\pi(\mu)c(\lambda-\mu,\mu)\\
                              &=&\sum_{\mu}{\bf a}\natural V_gf(\mu)\pi(\mu)=\pi_\Lambda({\bf a}){_\Lambda}\langle f,g\rangle. 
  \end{eqnarray*}  
  Therefore, the compatibility condition is actually a statement about the twisted convolution of $\big(V_gf(\lambda)\big)$ and ${\bf a}$ in $\ell^1_v(\Lambda,c)$.
  \item[(d)] ${_\Lambda}\langle f,g\rangle^*={_\Lambda}\langle g,f\rangle$ amounts to 
  \begin{eqnarray*}
    \big(\sum_{\lambda}\langle f,\pi(\lambda)g\rangle\pi(\lambda)\big)^*&=&\sum_{\lambda}\overline{\langle f,\pi(\lambda)g\rangle}\pi(\lambda)^*\\
                                                    &=&\sum_{\lambda}\overline{\langle f,\pi(\lambda)g\rangle}c(\lambda,\lambda)\pi(-\lambda)\\
                                                    &=&\sum_{\lambda}\overline{\langle f,\pi(-\lambda)g\rangle}c(\lambda,\lambda)\pi(\lambda)\\
                                                    &=&\sum_{\lambda}\overline{\langle \pi(\lambda)f,g\rangle}\pi(\lambda)\\
                                                    &=&\sum_{\lambda}\langle g,\pi(\lambda)f\rangle\pi(\lambda)={_\Lambda}\langle g,f\rangle.
  \end{eqnarray*}
  The previous argument is equivalent to the fact that the involution of $\big(V_gf(\lambda)\big)$ is $\big(V_fg(\lambda)\big)$ in $\ell^1(\Lambda,c)$.
  \item[(e)] The positivity of $_{\Lambda}\langle f,f\rangle$ for $f\in M^1_v(\mathbb{R}^d)$ in $C^*(\Lambda,c)$ is a non-trivial fact. It is a consequence of the Fundamental Identity of Gabor analysis, see \cite{ri88,lu07}. Recall, that the representation of $\mathcal{A}^1_v(\Lambda,c)$ is faithful on $L^2(\mathbb{R}^d)$. Therefore, it suffices to verify the positivity of ${_\Lambda}\langle f,f\rangle$ in $\mathcal{B}\big(L^2(\mathbb{R}^d)\big)$. Consequently, we have to check the positivity just for the dense subspace $M^1_v(\mathbb{R}^d)$:
  \begin{eqnarray*}
    \langle{_\Lambda}\langle f,f\rangle\cdot g,g\rangle&=&\sum_{\lambda\in\Lambda}\langle f,\pi(\lambda)f\rangle\langle \pi(\lambda)g,g\rangle\\
                                                       &=&{\mathrm{vol}(\Lambda)}^{-1}\sum_{\lambda^\circ\in\Lambda^\circ}\langle f,\pi(\lambda)g\rangle\langle\pi(\lambda)g,f\rangle\ge 0. 
  \end{eqnarray*}
  In the previous statement we invoked FIGA as in \cite{ri88}. The statements (a)-(e) yield that $M^1_v(\mathbb{R}^d)$ becomes a Hilbert $C^*(\Lambda,c)$-module when completed with respect to ${_\Lambda}\|f\|=\|{_\Lambda}\langle f,f\rangle\|^{1/2}_{\mathrm{op}}$. Since the ideal $\text{span}\{{_\Lambda}\langle f,g\rangle:f,g\in M^1_v(\mathbb{R}^d)\}$ is dense in $C^*(\Lambda,c)$, the Hilbert $C^*(\Lambda,c)$-module is full.   
\end{enumerate} 
\end{proof}
Suppose $\mathcal{A}$ is a unital $C^*$-algebra. If $({_\mathcal{A}}V,{_\mathcal{A}}\langle.,.\rangle)$ and  $({_\mathcal{A}}W,{_\mathcal{A}}\langle.,.\rangle)$ are left Hilbert $\mathcal{A}$-modules, then a map $T:{_\mathcal{A}}V\to{_\mathcal{A}}W$ is {\it adjointable}, if there is a map $T^*:{_\mathcal{A}}W\to{_\mathcal{A}}V$ such that
\begin{equation*}
 {_\mathcal{A}}\langle Tf,g\rangle={_\mathcal{A}}\langle f,T^*g\rangle~~\text{for all}~~f,g\in{_\mathcal{A}}V. 
\end{equation*}
We denote the set of all adjointable maps from ${_\mathcal{A}}V$ to ${_\mathcal{A}}W$ by $\mathcal{L}({_\mathcal{A}}V,{_\mathcal{A}}W)$. 
\par
If we view $C^*(\Lambda,c)$ as a full left Hilbert $C^*(\Lambda,c)$-module, then the map $\mathcal{C}^\Lambda_gf:={_\Lambda}\langle f,g\rangle$ is an adjointable operator from ${_\Lambda}V$ to $C^*(\Lambda,c)$ and its adjoint is the map $\mathcal{D}^\Lambda_g({\bf a}):=\pi_\Lambda({\bf a})\cdot g$.
More precisely, $\mathcal{A}^1_v(\Lambda,c)$ is a left inner product $\mathcal{A}^1_v(\Lambda,c)$-module with respect to $\pi_\Lambda({\bf a})\cdot \pi_\Lambda({\bf b})=\pi_{\Lambda}({\bf a})\pi_{\Lambda}({\bf b})$ and $_{C^*(\Lambda,c)}\langle \pi_\Lambda({\bf a}),\pi_\Lambda({\bf b})\rangle=\pi_{\Lambda}({\bf a})\pi_{\Lambda}({\bf b})^*$ for ${\bf a},{\bf b}\in\ell^1_v(\Lambda,c)$ and the module-norm of $\pi_\Lambda({\bf a})$ equals the operator norm of $\pi_\Lambda(\bf a)$. If we complete the inner product $\mathcal{A}^1_v(\Lambda,c)$-module with respect to this norm,  
then we obtain a full left Hilbert $C^*(\Lambda,c)$-module $_{C^*(\Lambda,c)}V$.
\begin{lemma}
  The map $\mathcal{C}^\Lambda_g$ is an element of $\mathcal{L}({_\Lambda}V_{C^*(\Lambda,c)}V)$ and $\mathcal{D}^\Lambda_g$ is in $\mathcal{L}(_{C^*(\Lambda,c)}V,{_\Lambda}V)$. Furthermore, $\mathcal{C}^\Lambda_g$ and $\mathcal{D}^\Lambda_g$ are adjoints of each other.   
\end{lemma}
\begin{proof}
  By the faithfulness of the representation of $C^*(\Lambda,c)$ it suffices to check the statement for the dense subalgebra $\mathcal{A}^1_v(\Lambda,c)$. Let ${\bf a}\in\ell^1_v(\Lambda,c)$ and $f,g\in M^1_v(\mathbb{R}^d)$. Then we have, that
  \begin{equation*}
    _{C^*(\Lambda,c)}\langle \pi_\Lambda({\bf a}),\mathcal{C}^\Lambda_gf\rangle=\pi_{\Lambda}({\bf a}){_\Lambda}\langle g,f\rangle={_\Lambda}\langle \pi_\Lambda({\bf a})\cdot g,f\rangle= {_\Lambda}\langle \mathcal{D}^\Lambda_g({\bf a}),f\rangle. 
  \end{equation*}
\end{proof}
The preceding lemma is a Hilbert $C^*$-module analog of the well-known fact that the {\it coefficient mapping} $C_{g,\Lambda}$ and the synthesis mapping $D_{g,\Lambda}$ are adjoint operators for a Gabor system $\mathcal{G}(g,\Lambda)$, where $C_{g,\Lambda}f:=\big(\langle f,\pi(\lambda)g\rangle\big)_\lambda$ is a map from $L^2(\mathbb{R}^d)$ to $\ell^2(\Lambda)$ and the synthesis mapping is defined by $D_{g,\Lambda}{\bf a}=\sum_{\lambda\in\Lambda}a(\lambda)\pi(\lambda)g$ for ${\bf a}\in\ell^2(\Lambda)$ and maps $\ell^2(\Lambda)$ into $L^2(\mathbb{R}^d)$. Therefore, the mappings $\mathcal{C}^\Lambda_g$ and $\mathcal{D}^\Lambda_g$ are noncommutative analogs of the coefficient and synthesis mappings of a Gabor system. In the Hilbert space setting a central role is played by the {\it Gabor frame operator} $S_{g,\Lambda}=D_{g,\Lambda}\circ C_{g,\Lambda}$, i.e. 
\begin{equation*}
  S_{g,\Lambda}f=\sum_{\lambda\in\Lambda}\langle f,\pi(\lambda)g\rangle\pi(\lambda)g~~\text{for}~~f\in L^2(\mathbb{R}^d).
\end{equation*}
Analogously we define the noncommutative frame operator $\mathcal{S}_{g}^\Lambda$ as the composition 
$\mathcal{D}^\Lambda_g\circ\mathcal{C}^\Lambda_g$, which is by definition a $C^*(\Lambda,c)$-module map. If $f,g\in M^1_v(\mathbb{R}^d)$, then
\begin{equation}
  S_{g,\Lambda}f={_\Lambda}\langle f,g\rangle\cdot g=\pi_{\Lambda}(V_gf)\cdot g.
\end{equation}
In other words, the Gabor frame operator on $M^1_v(\mathbb{R}^d)$ may be considered as a Hilbert $C^*(\Lambda,c)$-module map. Furthermore, the Gabor frame operator is a so-called {\it rank one Hilbert $C^*(\Lambda,c)$-module operator}. Recall, that on a left Hilbert $C^*$-module $({_\mathcal{A}}V,{_\mathcal{A}}\langle.,.\rangle)$ a {\it rank-one operator} $\Theta_{g,h}^{\mathcal{A}}$ is defined by $\Theta_{g,h}^\mathcal{A} f:={_\mathcal{A}}\langle f,g\rangle h$. Consequently, $S_{g,\Lambda}f$ is the rank-one operator $\Theta_{g,g}^\Lambda f$. A general rank-one operator 
$\Theta_{g,h}^\Lambda$ is given by 
\begin{equation}
  \Theta_{g,h}^\Lambda f=\sum_{\lambda\in\Lambda}\langle f,\pi(\lambda)g\rangle\pi(\lambda)h~~\text{for}~~f,g,h\in {_\Lambda}V,
\end{equation} 
which in Gabor analysis are called {\it Gabor frame type operators} and denoted by $S_{g,h,\Lambda}$. In the next section we will have to deal with finite sums of rank-one operators in our description of projective modules over $C^*(\Lambda,c)$. At the moment we want to take a closer look at adjointable operators on ${_\Lambda}V$. By definition, a map $T$ on ${_\Lambda}V$ is adjointable if there exists a map $T^*$ on ${_\Lambda}V$ such that 
\begin{equation*}
  {_\Lambda}\langle Tf,g\rangle={_\Lambda}\langle f,T^*g\rangle,~~f,g\in{_\Lambda}V.
\end{equation*} 
More explicitly, the last equation amounts to
\begin{equation*}
  \sum_{\lambda\in\Lambda}\langle Tf,\pi(\lambda)g\rangle=\sum_{\lambda\in\Lambda}\langle f,\pi(\lambda)T^*g\rangle.
\end{equation*}
If we restrict our interest to elements of the inner product $\mathcal{A}^1_v(\Lambda,c)$-module, then an adjointable $C^*(\Lambda,c)$-module map is bounded on $\ell^1_v(\Lambda)$, because every adjointable module map is bounded and the operator norm of the module map can be controlled by the $\ell^1_v$-norm. 
\par
Rieffel made the following crucial observation in \cite{ri88}, that $C^*(\Lambda,c)$ and the opposite algebra of $C^*(\Lambda^\circ,c)$ are closely related, namely they are Rieffel-Morita equivalent. We recall the notion of Rieffel-Morita equivalence for $C^*$-algebras after the discussion of right Hilbert $C^*$-modules over the opposite algebra of $C^*(\Lambda^\circ,c)$. Note that opposite time-frequency shifts $\pi(x,\omega)^{\rm{op}}$ are given by $T_xM_{\omega}$, which satisfy $T_xM_{\omega}=e^{-2\pi ix\cdot\omega}M_\omega T_x=e^{-2\pi ix\cdot\omega}\pi(x,\omega)$. 
\begin{lemma}
  The opposite algebra of $C^*(\Lambda^\circ,c)$ is $C^*(\Lambda^\circ,\overline{c})$.
\end{lemma}


The Theorem \ref{thm:HilbertModulesLeft} gives the completion of $M^1_v(\mathbb{R}^d)$ the structure of a left Hilbert $C^*(\Lambda^\circ,c)$-module $_{\Lambda^\circ}V$ with respect to the left action and $C^*(\Lambda^\circ,c)$-valued inner product $\langle .,.\rangle_{\Lambda^\circ}$ defined above, but we need a right module structure. There is a well-known procedure, which we describe in the following Lemma \ref{oppact}. 
\par
Let $\mathcal{A}$ be a $C^*$-algebra and $\mathcal{A}^{\rm{op}}$ its opposite $C^*$-algebra. Furthermore, we denote by $V^{\rm{op}}$ the {\it opposite} vector space structure on a Banach (Frechet) space $V$. We have a one-one correspondence between $\mathcal{A}$-left modules $V$ and $\mathcal{A}^{\rm{op}}$-right modules $V^{\rm{op}}$. 
\begin{lemma}\label{oppact}
  Let $\mathcal{A}$ be a $C^*$-algebra and $({_\mathcal{A} V},{_\mathcal{A}}\langle.,.\rangle)$ a left Hilbert $\mathcal{A}$-module. Then the opposite module $V^{\rm{op}}$ is a right Hilbert module for the opposite algebra $\mathcal{A}^{\rm{op}}$ with the $\mathcal{A}^{\rm{op}}$-valued inner product $\langle.,.\rangle_{\mathcal{A}^{\rm{op}}}:V^{\rm{op}}\times V^{\rm{op}}\to\mathcal{A}^{\rm{op}}$ given by $(f^{\rm{op}},g^{\rm{op}})\mapsto {_\mathcal{A}}\langle g,f\rangle^{\rm{op}}$. 
\end{lemma}
\begin{proof}
  Let $f^{\rm{op}},g^{\rm{op}}\in V^{\rm{op}}$ and $A^{\rm{op}}\in\mathcal{A}^{\rm{op}}$. Then $\langle.,.\rangle_{\mathcal{A}^{\rm{op}}}$ is compatible with the right action of $V^{\rm{op}}$:
  \begin{eqnarray*}
    \langle f^{\rm{op}},g^{\rm{op}}A^{\rm{op}}\rangle_{\mathcal{A}^{\rm{op}}}={_\mathcal{A}}\langle f^{\rm{op}},(Ag)^{\rm{op}}\rangle^{\rm{op}}&=&{_\mathcal{A}}\langle Ag,f\rangle\\&=&
    A{_\mathcal{A}}\langle g,f\rangle={_\mathcal{A}}\langle g,f\rangle^{\rm{op}}A^{\rm{op}}=\langle f^{\rm{op}},g^{\rm{op}}\rangle_{\mathcal{A}^{\rm{op}}}A^{\rm{op}}.
  \end{eqnarray*}
Since the compact $\mathcal{A}$-module operators are defined in terms of rank-one operators, we have to demonstrate that $\Theta_{g,h}^{A}f=\Theta_{h^{\rm{op}},g^{\rm{op}}}^{A^{\rm{op}}}f^{\rm{op}}$. By definition we have that 
\begin{eqnarray*}
  \Theta_{g,h}^{A}f={_\mathcal{A}}\langle f,h\rangle g=g^{\rm{op}}{_\mathcal{A}}\langle f,h\rangle=g^{\rm{op}}\langle h^{\rm{op}},f^{\rm{op}}\rangle=\Theta_{h^{\rm{op}},g^{\rm{op}}}^{A^{\rm{op}}}f^{\rm{op}}.
\end{eqnarray*}  
\end{proof}
Therefore, Lemma \ref{oppact} gives the following right action of $\mathcal{A}^1_v(\Lambda^\circ,\overline{c})$ on $M^1_v(\mathbb{R}^d)$ by
\begin{equation}\label{eq:LeftAction}
  f\cdot\pi_{\Lambda^\circ}({\bf b})={\text{vol}(\Lambda)}^{-1}\sum_{\lambda^\circ\in\Lambda^\circ}\pi(\lambda^\circ)^*f~{b}(\lambda^\circ).  
\end{equation}  
and $C^*(\Lambda,\overline{c})$-valued inner product $\langle.,.\rangle_{\Lambda^\circ}$:
\begin{equation*}
\langle f,g\rangle_{\Lambda^\circ}=
{\text{vol}(\Lambda)}^{-1}\sum_{\lambda^\circ\in\Lambda^\circ}\pi(\lambda^\circ)^*\langle \pi(\lambda^\circ)g,f\rangle.
\end{equation*}
We summarize all these observations and statements in the following theorem:
\begin{theorem}\label{thm:HilbertModulesRight}
Let $\Lambda$ be a lattice in $\mathbb{R}^{2d}$. If $v$ is a submultiplicative weight, then the completion of $M^1_v(\mathbb{R}^{d})$ becomes a full right Hilbert   $C^*(\Lambda^\circ,\overline{c})$-module $V_{\Lambda^\circ}$ for the right action of $\mathcal{A}^1_v(\Lambda^\circ,\overline{c})$ on $M^1_v(\mathbb{R}^{d})$
\begin{equation*}
  g\cdot\pi_{\Lambda^\circ}({\bf b})={\mathrm{vol}(\Lambda)}^{-1}\sum_{\lambda^\circ\in\Lambda^\circ}\pi(\lambda^\circ)^*g~{b}(\lambda^\circ)~~\text{for}~~{\bf b}=\big(b(\lambda^\circ)\big)\in\ell^1_v(\Lambda^\circ),g\in M^1_v(\mathbb{R}^{d}),		
\end{equation*}
with the $C^*(\Lambda^\circ,\overline{c})$-inner product
\begin{equation*}
  \langle f,g\rangle_{\Lambda^\circ}={\mathrm{vol}(\Lambda)}^{-1}\sum_{\lambda^\circ\in\Lambda^\circ}\pi(\lambda^\circ)^*\langle g,\pi(\lambda^\circ)f\rangle~~\text{for}~~f,g\in M^1_v(\mathbb{R}^{d})
\end{equation*}
when completed with respect to the norm $\|f\|_{\Lambda^\circ}=\|\langle f,f\rangle_{\Lambda^\circ}\|_{\mathrm{op}}^{1/2}$. 
\end{theorem}
Rieffel introduced in \cite{ri74-2} the notion of strong Morita equivalence for $C^*$-algebras, which we state in the following definition.
\begin{definition}\label{Def:Morita}
 Let $\mathcal{A}$ and $\mathcal{B}$ be $C^*$-algebras. Then an $\mathcal{A}$-$\mathcal{B}$-{\it equivalence bimodule} ${_\mathcal{A}}V{_\mathcal{B}}$ is an $\mathcal{A}$-$\mathcal{B}$-bimodule such that:
 \begin{enumerate}
   \item[(a)] ${_\mathcal{A}}V{_\mathcal{B}}$ is a full left Hilbert $\mathcal{A}$-module and a full right Hilbert $\mathcal{B}$-module;
   \item[(b)] for all $f,g\in {_\mathcal{A}}V{_\mathcal{B}}, A\in\mathcal{A}$ and $B\in\mathcal{B}$ we have that
     \begin{equation*}
       \langle A\cdot f,g\rangle{_\mathcal{B}}=\langle f,A^*\cdot g\rangle{_\mathcal{B}}~~\text{and}~~{_\mathcal{A}}\langle f\cdot B,g\rangle={_\mathcal{A}}\langle f,g\cdot B^*\rangle;
     \end{equation*}
    \item[(c)] for all $f,g,h\in {_\mathcal{A}}V{_\mathcal{B}}$,
       \begin{equation*}
       {_\mathcal{A}}\langle f,g\rangle\cdot h=f\cdot\langle g,h\rangle{_\mathcal{B}}.
     \end{equation*}
 \end{enumerate}
 The $C^*$-algebras $\mathcal{A}$ and $\mathcal{B}$ are called {\it Rieffel-Morita equivalent} if there exists an $\mathcal{A}-\mathcal{B}$ equivalence bimodule.
\end{definition}  
In words, Condition (b) in Definition \ref{Def:Morita} says that $\mathcal{A}$ acts by adjointable operators on $V_{\mathcal{B}}$ and that $\mathcal{B}$ acts by adjointable operators on $_{\mathcal{A}}V$, and Condition (c) in Definition \ref{Def:Morita} is an associativity condition between the $\mathcal{A}$-inner product and the $\mathcal{B}$-inner product.
\par
The Theorems \ref{thm:HilbertModulesLeft} and \ref{thm:HilbertModulesRight} give an $C^*(\Lambda,c)$-$C^*(\Lambda^\circ,\overline{c})$ equivalence bimodule  ${_\Lambda}V_{\Lambda^\circ}$. The associativity condition between ${_\Lambda}\langle.,.\rangle$ and $\langle.,.\rangle_{\Lambda^\circ}$ is a statement about rank one Hilbert $C^*$-module operators for $C^*(\Lambda,c)$ and $C^*(\Lambda^\circ,\overline{c})$, which in Gabor analysis is known as the {\it Janssen representation} of a Gabor frame-type operator.
\begin{theorem}\label{thm:AssCond}
Let $\Lambda$ be a lattice in $\mathbb{R}^{2d}$. Then for all $f,g,h\in M^1_v(\mathbb{R}^d)$
\begin{equation}\label{eq:AssCond}
  {_\Lambda}\langle f,g\rangle\cdot h=f\cdot\langle g,h\rangle_{\Lambda^\circ},
\end{equation}   
or in terms of Gabor frame-type operators:
\begin{equation}
  S_{g,h,\Lambda}f={\mathrm{vol}(\Lambda)}^{-1} S_{h,f,\Lambda^\circ}g.
\end{equation}
\end{theorem}
\begin{proof}
 The identity \eqref{eq:AssCond} is equivalent to 
 \begin{equation*}
   \big\langle{_\Lambda}\langle f,g\rangle\cdot h,k\big\rangle=\big\langle f\cdot\langle g,h\rangle_{\Lambda^\circ},k\big\rangle
 \end{equation*}
 for all $k\in M^1_v(\mathbb{R}^d)$. More explicitly, the associativity condition reads as follows
 \begin{equation*}
   \sum_{\lambda\in\Lambda}\langle f,\pi(\lambda)g\rangle\langle\pi(\lambda)h,k\rangle={\mathrm{vol}(\Lambda)}^{-1}\sum_{\lambda^\circ\in\Lambda^\circ}\langle f,\pi(\lambda^\circ)k\rangle\langle \pi(\lambda^\circ)h,g\rangle.
 \end{equation*}
In other words, the associativity condition is the Fundamental Identity of Gabor analysis. Therefore,  Theorem \ref{FIGA} gives the desired result.
\end{proof}  
The observation, that the associativity condition for ${_\Lambda}\langle.,.\rangle$ and $\langle.,.\rangle_{\Lambda^\circ}$ is the Fundamental Identity of Gabor analysis, allows one to link projective modules over noncommutative tori and Gabor frames for modulation spaces. 
\par
The last step in the construction of an equivalence bimodule between $C^*(\Lambda,c)$ and $C^*(\Lambda^\circ,\overline{c})$ is to establish that $C^*(\Lambda,c)$ acts by adjointable maps on $C^*(\Lambda^\circ,\overline{c})$, which in the present setting is a non-trivial task. The main difficulty lies in the fact, that we actually have just a pre-equivalence bimodule. Therefore, in addition to the Condition (b) in Definition \ref{Def:Morita} one has to check that the actions are bounded:  
\begin{equation*}
  \langle \pi_{\Lambda}({\bf a})\cdot g,\pi_{\Lambda}({\bf a})\cdot g\rangle_{\Lambda^\circ}\le \|\pi_{\Lambda}({\bf a})\|^2_{\rm{op}}~\|g\|_{\Lambda^\circ}
\end{equation*}   
and 
\begin{equation*}
  _{\Lambda}\langle g\cdot \pi_{\Lambda}({\bf b}),g\cdot \pi_{\Lambda}({\bf b})\rangle\le \|\pi_{\Lambda}({\bf b})\|^2_{\rm{op}}~ {_\Lambda}\|g\|
\end{equation*}   
for all ${\bf a}\in\ell^1_v(\Lambda),{\bf b}\in\ell^1_v(\Lambda^\circ)$ and $g\in M^1_v(\mathbb{R}^d)$. These inequalities are formulated in Proposition 2.14 in \cite{ri88} and the proof of these inequalities holds also in the present context. In words, the first inequality yields the boundedness of the left action of $\mathcal{A}^1_v(\Lambda,c)$ on the right Hilbert $C^*$-module $V_{\Lambda^\circ}$ and the second inequality amounts to an analogous statement for the right action of $\mathcal{A}^1_v(\Lambda^\circ,\overline{c})$ on the left Hilbert $C^*$-module ${_\Lambda}V$.
\par
Therefore, we have that the completion of $M^1_v(\mathbb{R}^d)$ with respect to ${_\Lambda}\|f\|=\|{_\Lambda}\langle f,f\rangle\|^{1/2}_{\rm{op}}$ or equivalently by  $\|f\|{_{\Lambda^\circ}}=\|\langle f,f\rangle{_{\Lambda^\circ}}\|^{1/2}_{\rm{op}}$, becomes an equivalence bimodule ${_\Lambda}V_{\Lambda^\circ}$ between $C^*(\Lambda,c)$ and $C^*(\Lambda^\circ,\overline{c})$. We summarize the previous discussion in the following theorem, which includes one of the main results in \cite{ri88}.
\begin{theorem}
Let $\Lambda$ be a lattice in $\mathbb{R}^{2d}$. Then the completion of $M^1_v(\mathbb{R}^d)$ with respect to ${_\Lambda}\|f\|=\|{_\Lambda}\langle f,f\rangle\|^{1/2}_{\rm{op}}$ becomes an equivalence bimodule ${_\Lambda}V_{\Lambda^\circ}$ between $C^*(\Lambda,c)$ and $C^*(\Lambda^\circ,\overline{c})$.  
\end{theorem}
In the present setting we do not just have an equivalence bimodule ${_\Lambda}V_{\Lambda^\circ}$ between $C^*(\Lambda,c)$ and $C^*(\Lambda^\circ,\overline{c})$, but we have dense subspaces $M^1_v(\mathbb{R}^d)$ of ${_\Lambda}V_{\Lambda^\circ}$ that give rise to equivalence bimodules between $\mathcal{A}^1_v(\Lambda,c)$ and $\mathcal{A}^1_v(\Lambda^\circ,\overline{c})$. In Connes' work on noncommutative geometry this kind of structure is very important, see \cite{co80}. Connes discussed the general framework in \cite{co81-2}. For further motivation and results we refer the interested reader to \cite{co94-1}. We briefly recall Connes's general result, see also Proposition 3.7 and its proof in \cite{ri88}. 
\begin{theorem}[Connes]\label{Connes}
  Let $\mathcal{A}$ and $\mathcal{B}$ be unital $C^*$-algebras that are Morita equivalent via an equivalence bimodule ${_\mathcal{A}V_{\mathcal{B}}}$. Suppose we have dense $*$-Banach (or Frechet) subalgebras $\mathcal{A}_0$ and $\mathcal{B}_0$ of $\mathcal{A}$ and $\mathcal{B}$ respectively containing the identity elements. Furthermore we assume that $\mathcal{A}_0$ and $\mathcal{B}_0$ are spectrally invariant in $\mathcal{A}$ and $\mathcal{B}$ respectively. Let $V_0$ be a dense subspace of ${_\mathcal{A}V_{\mathcal{B}}}$ which is closed under the actions of $\mathcal{A}_0$ and $\mathcal{B}_0$, and such that the restrictions of the inner products ${_\mathcal{A}}\langle.,.\rangle$ and $\langle.,.\rangle_{\mathcal{B}}$ have values in $\mathcal{A}_0$ and $\mathcal{B}_0$ respectively. Then $V_0$ is a finitely generated projective left $\mathcal{A}_0$-module and the mapping from $\mathcal{A}\otimes_{\mathcal{A}_0}V_0$ to ${_\mathcal{A}V_{\mathcal{B}}}$ defined by $A\otimes f\mapsto Af$ is an isomorphism of left $\mathcal{A}-modules$. In addition we have that $V_0$ is a finitely generated projective right $\mathcal{B}_0$-module and the mapping from $V_0\otimes_{\mathcal{B}_0}\mathcal{B}$ to ${_\mathcal{A}V_{\mathcal{B}}}$ defined by $f\otimes B\mapsto fB$ is an isomorphism of left $\mathcal{B}-modules$. Therefore $V_0$ is an equivalence bimodule between $\mathcal{A}_0$ and $\mathcal{B}_0$.
\end{theorem}
The result of Connes suggests the following definition. If we are in the situation of Theorem \ref{Connes}, then we call the algebras $\mathcal{A}_0$ and $\mathcal{B}_0$ Rieffel-Morita equivalent.
\par
Now, we are in the position to draw an important consequence concerning the structure of the equivalence bimodule ${_\Lambda}V_{\Lambda^\circ}$ from the spectral invariance of $\mathcal{A}^1_v(\Lambda^\circ,\overline{c})$ in $C^*(\Lambda^\circ,\overline{c})$
\begin{theorem}\label{thm:MoritaNCWiener}
  Let $\Lambda$ be a lattice in $\mathbb{R}^{2d}$. Then the noncommutative Wiener algebras $\mathcal{A}^1_v(\Lambda,c)$ and $\mathcal{A}^1_v(\Lambda^\circ,c)$ are Morita-equivalent through $M^1_v(\mathbb{R}^d)$ if and only if $v$ is a GRS-weight. Consequently, $M^1_v(\mathbb{R}^d)$ is a finitely generated projective left $\mathcal{A}^1_v(\Lambda,c)$-module and a finitely generated right $\mathcal{A}^1_v(\Lambda^\circ,\overline{c})$-module.  
\end{theorem}

\begin{proof} 
  We follow closely the discussion of Proposition $3.7$ in \cite{ri88}. Let $\mathcal{A}=C^*(\Lambda,c)$ and $\mathcal{B}=C^*(\Lambda^\circ,\overline{c})$. Let ${_\Lambda}V$ be the projective left $C^*(\Lambda,c)$-module completion of $M^1_v(\mathbb{R}^d)$ with $C^*(\Lambda,c)$-valued inner product 
  \begin{equation*}
    {_\Lambda}\langle f,g\rangle=\sum_{\lambda\in\Lambda}\langle f,\pi(\lambda)g\rangle\pi(\lambda)~~\textit{for}~~f,g\in M^1_v(\mathbb{R}^d).
  \end{equation*}
  Furthermore, we have the $C^*(\Lambda^\circ,\overline{c})$-valued inner product 
  \begin{equation*}
    \langle f,g\rangle_{\Lambda}={\mathrm{vol}(\Lambda)}^{-1}\sum_{\lambda^\circ\in\Lambda^\circ}\pi^*(\lambda^\circ)\langle \pi(\lambda^\circ)g,f\rangle~~\textit{for}~~f,g\in M^1_v(\mathbb{R}^d).
  \end{equation*}   
  Now, we consider the dense involutive unital subalgebras $\mathcal{A}_0=\mathcal{A}^1_v(\Lambda,c)$ and $\mathcal{B}_0=\mathcal{A}^1_v(\Lambda^\circ,\overline{c})$ of $C^*(\Lambda,c)$ and $C^*(\Lambda^\circ,\overline{c})$, respectively. Then $M^1_v(\mathbb{R}^d)$ is a dense subspace of ${_\Lambda}V$ that is closed under the actions of $\mathcal{A}^1_v(\Lambda,c)$ and $\mathcal{A}^1_v(\Lambda^\circ,\overline{c})$ given by
  \begin{equation*}
    \pi_\Lambda({\bf a})\cdot f=\sum_{\lambda\in\Lambda}a(\lambda)\pi(\lambda)f~~\textit{for}~~{\bf a}\in\ell^1_v(\Lambda), f\in M^1_v(\mathbb{R}^d)
  \end{equation*}
  and 
  \begin{equation*}
    \pi_{\Lambda^\circ}({\bf b})\cdot f={\mathrm{vol}(\Lambda)}^{-1}\sum_{\lambda^\circ\in\Lambda^\circ}\pi(\lambda^\circ)^*f~~\textit{for}~~{b(\lambda^\circ)}~~{\bf b}\in\ell^1_v(\Lambda^\circ), f\in M^1_v(\mathbb{R}^d).
  \end{equation*}
  Furthermore, the restrictions of the inner products ${_\Lambda}\langle .,.\rangle$ and $\langle .,.\rangle_{\Lambda^\circ}$ to $M^1(\mathbb{R}^d)$ have values in $\mathcal{A}^1_v(\Lambda,c)$ and $\mathcal{A}^1_v(\Lambda^\circ,\overline{c})$, respectively. The final ingredient in our proof is the spectral invariance of $\mathcal{A}^1_v(\Lambda^\circ,\overline{c})$ in $C^*(\Lambda^\circ,\overline{c})$, which is equivalent to $v$ being a GRS-weight by Theorem \ref{grle}. An application of Proposition 3.7 in \cite{ri88} gives the desired assertion that $M^1_v(\mathbb{R}^d)$ is an equivalence bimodule between  $\mathcal{A}^1_v(\Lambda,c)$ and $\mathcal{A}^1_v(\Lambda^\circ,\overline{c})$. 
\end{proof}
All the results in this section hold for $\mathcal{A}^1_{v^s}(\Lambda,c), \mathcal{A}^1_{v^s}(\Lambda^\circ,\overline{c})$ and $M^1_{v^s}(\mathbb{R}^d)$. Therefore all theorems remain true for $\mathcal{A}^\infty_{v^s}(\Lambda,c), \mathcal{A}^\infty_{v^s}(\Lambda^\circ,\overline{c})$ and $\mathscr{S}_{v^s}(\mathbb{R}^d)$. By construction $\mathscr{S}_v(\mathbb{R}^d)$ is a projective limit of $M^1_v(\mathbb{R}^d)$; consequently statements for $M^1_v(\mathbb{R}^d)$ translate into ones about $\mathscr{S}_v(\mathbb{R}^d)$.  Consequently, the preceding results allow us to prove the Morita equivalence of $\mathcal{A}^\infty_v(\Lambda,c)$ and $\mathcal{A}^\infty_v(\Lambda^\circ,\overline{c})$.    
\begin{theorem}\label{thm:MoritaNCTori1}
Let $\Lambda$ be a lattice in $\mathbb{R}^{2d}$. Then $\mathcal{A}^\infty_v(\Lambda,c)$ and $\mathcal{A}^\infty_v(\Lambda^\circ,\overline{c})$ are Morita equivalent through the equivalence bimodule $\mathscr{S}_v(\mathbb{R}^d)$ if and only if $v$ is a GRS-weight. 
\end{theorem}
The case $v(x,\omega)=1+|x|^2+|\omega|^2$ for $(x,\omega)\in\mathbb{R}^2$ is the famous theorem of Connes in \cite{co80} about the Morita equivalence of smooth noncommutative tori $\mathcal{A}^\infty(\Lambda,c)$ and on the level of $C^*$-algebras the theorem was proved by Rieffel in 
\cite{ri81}. 

\section{Applications to Gabor analysis}
In the present section we link the results about projective modules over $\mathcal{A}^1_v(\Lambda^\circ,\overline{c})$ and $\mathcal{A}^\infty_v(\Lambda^\circ,\overline{c})$ with multi-window Gabor frames for modulation spaces. {\it Modulation spaces} \cite{fe83-4} were introduced by Feichtinger in $1983$. Later Feichtinger described modulation spaces in terms of Gabor frames \cite{fe89-1}. In collaboration with Gr\"ochenig he developed the coorbit theory \cite{fegr89}, which associates to an integrable representation of a locally compact group a class of function spaces. The coorbit spaces for the Schr\"odinger representation of the Heisenberg group is the class of modulation spaces. In the coorbit theory \cite{fegr89} modulation spaces are introduced as subspaces of the space of conjugate linear functionals $(M^1_v)^\neg(\mathbb{R}^d)$. Recall that a weight $m$ is called $v$-moderate, if there exists a constant $C>0$ such that $m(x+y,\omega+\eta)\le C v(x,\omega)m(y,\eta)$ for all $(x,\omega),(y,\eta)\in\mathbb{R}^{2d}$. Let $m$ be a $v$-moderate weight on $\mathbb{R}^{2d}$ and $g$ a non-zero window in $M^1_v(\mathbb{R}^d)$. Then the modulation spaces $M^{p,q}_m(\mathbb{R}^d)$ are defined as
\begin{equation*} M^{p,q}_m(\mathbb{R}^d)=\{f\in\mathscr{S}'_v(\mathbb{R}^d):\|f\|_{M^{p,q}_m}=\Big(\int_{\mathbb{R}^d}\Big(\int_{\mathbb{R}^d}|V_gf(x,\omega)|^pm(x,\omega)^pdx\Big)^{q/p}d\omega\Big)^{1/q}<\infty\},
\end{equation*}   
for $p,q\in[1,\infty]$. The definition of $M^{p,q}_m(\mathbb{R}^d)$ seems to dependent on the window function $g$, but it is a non-trivial fact that any other $g\in M^1_v(\mathbb{R}^d)$ defines the same space \cite{fe83-4,fegr89}. Furthermore the norm $\|f\|_{M^{p,q}_m}$ depends on the chosen window function $g$, but any other $g\in M^1_v(\mathbb{R}^d)$ defines an equivalent norm on $M^{p,q}_m(\mathbb{R}^d)$, see Chapter 11 in \cite{gr01} and Proposition 11.3.2 for an extensive discussion of these matters. We continue with stating some properties of modulation spaces. The modulation space $M^{p,q}_m(\mathbb{R}^d)$ is a Banach space, which is invariant under time-frequency shifts. The growth of the $v$-moderate weight $m$ allows to draw some conclusions about $M^{p,q}_m(\mathbb{R}^d)$: (i) if $v$ grows atmost polynomially, then $M^{p,q}_m(\mathbb{R}^d)$ are subspaces of the class of tempered distributions $\mathscr{S}(\mathbb{R}^d)$; (ii) suppose $v$ grows at most sub-exponentially, then $M^{p,q}_m(\mathbb{R}^d)$ are subspaces of the ultra distributions of Bj\"orck and Komatsu; (iii) if $v$ grows exponentially, then $M^{p,q}_m(\mathbb{R}^d)$ are subspaces of the Gelfand-Shilov space $(S_{{\scriptstyle\frac{1}{2},\frac{1}{2}}})'(\mathbb{R}^d)$. We refer the reader to Feichtinger's survey article \cite{fe06} for a discussion of the properties, applications of modulation spaces, and an extensive list of references.
\par
In the last two decades modulation spaces have found various applications in time-frequency analysis and especially Gabor analysis. For example 
the existence of a Janssen representation for Gabor frames $\mathcal{G}(g,\Lambda)$ with $g\in M^1_v(\mathbb{R}^d)$ is one of the most important results in Gabor analysis \cite{fegr97}.  The proof of this result relies on the restriction property of functions in $M^1_v(\mathbb{R}^d)$ to lattices $\Lambda^\circ$ in $\mathbb{R}^{2d}$, i.e. for $g\in M^1_v(\mathbb{R}^d)$ the sequence $(\langle g,\pi(\lambda^\circ)g\rangle)_{\lambda^\circ}$ is in $\ell^1_v(\Lambda^\circ)$.
\begin{proposition}
Let $\Lambda$ be a lattice in $\mathbb{R}^{2d}$ and $\mathcal{G}(g,\Lambda)$ be a Gabor frame for $L^2(\mathbb{R}^d)$ with $g\in M^1_v(\mathbb{R}^d)$. Then the Janssen representation of the Gabor frame operator 
\begin{equation}\label{JansRep}
  S_{g,\Lambda}=\mathrm{vol}(\Lambda)^{-1}\sum_{\lambda^\circ\in\Lambda^\circ}\langle g,\pi(\lambda^\circ)g\rangle\pi(\lambda^\circ)
\end{equation}
converges absolutely in the operator norm and it defines an element of $\mathcal{A}^1_v(\Lambda^\circ,c)$.
\end{proposition}
Note that $\mathcal{G}(g,\Lambda)$ is a Gabor frame for $L^2(\mathbb{R}^d)$ if and only if the Janssen representation of its Gabor frame operator is invertible on $\mathcal{B}(L^2(\mathbb{R}^d))$. By the spectral invariance of $\mathcal{A}^1_v(\Lambda^\circ,c)$ in $C^*(\Lambda,c)$ for $v$ a GRS-weight the inverse of $S_{g,\Lambda}$ for $g\in M^1_v(\mathbb{R}^d)$ is again an element of $\mathcal{A}^1_v(\Lambda^\circ,c)$ and the canonical tight Gabor atom $S^{-1}_{g,\Lambda}g$ is in $M^1_v(\mathbb{R}^d)$. Therefore the Janssen representation of the Gabor frame operator is of most importance for the discussion of Gabor frames with good Gabor atoms $g\in M^1_v(\mathbb{R}^d)$ or $g\in\mathscr{S}_v(\mathbb{R}^d)$, because it allows to construct reconstruction formulas with good synthesis windows \cite{grle04,gr07}. This discussion shows that the classes $M^1_v(\mathbb{R}^d)$ and $\mathscr{S}_v(\mathbb{R}^d)$ for $v$ a GRS-weight are good classes of Gabor atoms. In the following we will always assume that $v$ is a GRS-weight, if not stated explicitly.
\par
The preceding discussion and the results in Section $2$ about the topological stable rank of $\mathcal{A}^1_v(\Lambda^\circ,c)$ and $\mathcal{A}^\infty_v(\Lambda^\circ,c)$ enable a study of the deeper properties of the set of Gabor frames with atoms in $M^1_v(\mathbb{R}^d)$. In the seminal paper \cite{ri83-1} on the topological stable rank an interesting property of Banach algebras $\mathcal{A}$ with $\mathrm{tsr}(A)=1$ was noted, namely that its group of invertible elements is dense in $\mathcal{A}$. In \cite{blkuro92} the topological stable rank of completely noncommutative tori $C^*(\Lambda,c)$ was shown to be one. Therefore by the spectral invariance of $\mathcal{A}^1_v(\Lambda,c)$ and $\mathcal{A}^\infty_v(\Lambda,c)$ $v$ in $C^*(\Lambda,c)$ yields that $\mathrm{tsr}(\mathcal{A}^1_v(\Lambda^\circ,c))=\mathrm{tsr}(\mathcal{A}^\infty_v(\Lambda^\circ,c))=1$ for $\Lambda$ {\it completely irrational}. Recall that a lattice is {\it completely irrational} if for any $\lambda\in\Lambda$ there exists a $\mu\in\Lambda$ such that $\Omega(\lambda,\mu)$ is an irrational number.
\begin{theorem}
Let $\Lambda$ be completely irrational and $v$ a GRS-weight on $\mathbb{R}^{2d}$. Then the following holds: 
\begin{enumerate}
 \item The set of Gabor frames $\mathcal{G}(g,\Lambda)$ for $L^2(\mathbb{R}^d)$ with $g\in M^1_v(\mathbb{R}^d)$ is dense in $\mathcal{A}^1_v(\Lambda^\circ,c)$.
 \item The set of Gabor frames $\mathcal{G}(g,\Lambda)$ for $L^2(\mathbb{R}^d)$ with $g\in \mathscr{S}_v(\mathbb{R}^d)$ is dense in $\mathcal{A}^\infty_v(\Lambda^\circ,c)$.
\end{enumerate}
\end{theorem}
\begin{proof}
For $g\in M^1_v(\mathbb{R}^d)$ (or $g\in\mathscr{S}(\mathbb{R}^d)$) the Gabor frame operator $S_{g,\Lambda}$ has a Janssen representation in $\mathcal{A}^1_v(\Lambda^\circ,c)$ (or $\mathcal{A}^\infty_v(\Lambda^\circ,c)$), see Proposition \eqref{JansRep}. Note that $\mathcal{G}(g,\Lambda)$ is invertible if and only if the frame operator $S_{g,\Lambda}$ is invertible which is equivalent to the invertibility of the Janssen representation of $S_{g,\Lambda}$ in $\mathcal{A}^1_v(\Lambda,c)$ (or $\mathcal{A}^\infty_v(\Lambda^\circ,c)$). Consequently, the Gabor frame operator $S_{g,\Lambda}$ of $\mathcal{G}(g,\Lambda)$ for $g\in M^1_v(\mathbb{R}^d)$ (or $g\in\mathscr{S}(\mathbb{R}^d)$) is an invertible element in $\mathcal{A}^1_v(\Lambda^\circ,c)$ (or $\mathcal{A}^\infty_v(\Lambda^\circ,c)$). The assumption that $\Lambda$ is completely irrational and Badea's result \cite{ba98-4} imply that  $\mathrm{tsr}(\mathcal{A}^1_v(\Lambda^\circ,c))=\mathrm{tsr}(\mathcal{A}^\infty_v(\Lambda^\circ,c))=\mathrm{tsr}(C^*(\Lambda^\circ,c))=1$. Therefore the set of Gabor frames $\mathcal{G}(g,\Lambda^\circ)$ with $g\in M^1_v(\mathbb{R}^d)$ (or $g\in\mathscr{S}(\mathbb{R}^d)$) is dense in $\mathcal{A}^1_v(\Lambda^\circ,c)$ or ($\mathcal{A}^\infty_v(\Lambda^\circ,c)$).
\end{proof}
After this work was finished Prof. K.~Gr\"ochenig informed us that our main result about the existence of good multi-window Gabor frames might also follow from the coorbit theory \cite{fegr89} and his work in \cite{gr91}, but as far as we know this consequence of the coorbit theory has not been  published so far. Furthermore, the coorbit theory does not provide an explanation, why one just needs a finite number of Gabor atoms. In our approach this is reflected in the fact that the projective module $M^1_v(\mathbb{R}^d)$ is finitely generated. Therefore our results provide a link between the foundations of Gabor analysis and projective modules over noncommutative tori that reveals some new mathematical structures of Gabor frames. 
\par
The main result of this section is to demonstrate that the statements of Theorem \ref{thm:MoritaNCWiener} and of Theorem \ref{thm:MoritaNCTori1}  provide the existence of good multi-window Gabor frames for lattices in $\mathbb{R}^{2d}$ with Gabor atoms $g_i$ in $M^1_v(\mathbb{R}^d)$ and $\mathscr{S}(\mathbb{R}^d)$. The proof of this fact relies on the observation that a standard module frame for the finitely generated $C^*(\Lambda,c)$-left module ${_\Lambda}V$ is actually a tight multi-window Gabor frame for $L^2(\mathbb{R}^d)$.
\par
Projective modules over Hilbert $C^*$-algebras have a natural description in terms of module frames as was originally observed by Rieffel, e.g. in \cite{ri88}. Later Frank and Larson introduced module frames for arbitrary finitely and countably generated Hilbert $C^*$-modules in \cite{frla02}. In Theorem 5.9 in \cite{frla02} they present a formulation of Rieffel's observation in their framework. Namely, that any algebraically generating set of a finitely generated projective Hilbert $C^*$-module is a standard module frame. In the following we explore this statement for the Hilbert $C^*(\Lambda,c)$-module ${_\Lambda}V$. We start with Rieffel's reconstruction formula for elements $f$ of the finitely generated projective right $C^*(\Lambda,c)$-module ${_\Lambda}V$.  
\begin{proposition}[Rieffel]
Let $\Lambda$ be a lattice in $\mathbb{R}^{2d}$. Then there exist $g_1,...,g_n\in {_\Lambda}V$ such that 
\begin{equation}\label{eq:RecoModFrame}
  f=\sum_{i=1}^n{_\Lambda}\langle f,g_i\rangle\cdot g_i
\end{equation}
for all $f\in {_\Lambda}V$.
\end{proposition}
Recently, Frank and Larson emphasized in \cite{frla02} that the reconstruction formula \eqref{eq:RecoModFrame} is equivalent to the fact that $\{g_1,...,g_n\}$ is a {\it standard  tight module frame} for the finitely generated projective module ${_\Lambda}V$, i.e. for all $f\in {_\Lambda}V$ we have that
\begin{equation}\label{ModFrameCond}
  {_\Lambda}\langle f,f\rangle=\sum_{i=1}^n{_\Lambda}\langle g_i,f\rangle {_\Lambda}\langle f,g_i\rangle.
\end{equation}
By definition of the $C^*(\Lambda,c)$-valued inner product the conditions \eqref{ModFrameCond} takes the following explicit form:
\begin{equation}\label{ModFrameCondNCTori}
\sum_{\lambda\in\Lambda}\langle f,\pi(\lambda)f\rangle\pi(\lambda)=\sum_{i=1}^n\sum_{\lambda\in\Lambda} (V_{g_i}f\natural V_f{g_i})(\lambda)\pi (\lambda)~~\text{for all}~~f\in{_\Lambda}V.
\end{equation} 
Note that taking the trace $\mathrm{tr}_\Lambda$, $\mathrm{tr}_\Lambda(A)=a_0$ for $A=\sum_{\lambda\in\Lambda}a(\lambda)\pi(\lambda)$, of the module frame condition \eqref{ModFrameCondNCTori} yields 
\begin{equation}\label{MultiWindowFrame}
  \|f\|^2_2=\sum_{i=1}^n\sum_{\lambda\in\Lambda}|\langle f,\pi(\lambda)g_i\rangle|^2,
\end{equation}
which are known in Gabor analysis as {\it multi-window Gabor frames} \cite{zezi97}. Therefore a standard module frame $\{g_1,...,g_n\}$ for ${_\Lambda}V$ is a multi-window Gabor frame $\mathcal{G}(g_1,...g_n,\Lambda)=\mathcal{G}(g_1,\Lambda)\cup\cdots\mathcal{G}(g_n,\Lambda)$ for $L^2(\mathbb{R}^d)$. The {\it multi-window Gabor frame operator} $S_\Lambda$ associated to a multi-window Gabor system $\mathcal{G}(g_1,...g_n;\Lambda)$ is given by 
\begin{equation}
  S_\Lambda f=\sum_{i=1}^n S_{g_i,\Lambda}f~~\text{for}~~f\in L^2(\mathbb{R}^d).
\end{equation} 
The operator $S_\Lambda$ is the finite-rank ${_\Lambda}V$-module operator $S_\Lambda=\sum_{j=1}^n\Theta_{g_j,g_j}^\Lambda$. Note that $S_\Lambda$ is positive bounded ${_\Lambda}V$ module map operator and \eqref{ModFrameCond} means that $S_\Lambda$ is invertible on ${_\Lambda}V$.
We summarize these observations in the following theorem that links the abstract notion of standard module frames over noncommutative tori with the notion of multi-window Gabor frames due to the engineers Zibulski and Zeevi \cite{zezi97}.
\begin{theorem}\label{ThmMultiWindow}
  Let $\Lambda$ be a lattice in $\mathbb{R}^{2d}$. Then a standard module frame $\{g_1,...,g_n\}$ for ${_\Lambda}V$ is a tight multi-window Gabor frame $\mathcal{G}(g_1,...g_n;\Lambda)$ for $L^2(\mathbb{R}^d)$.
\end{theorem}
\begin{proof}
  First we want to check that the module frame condition \eqref{ModFrameCond} holds for all $f\in L^2(\mathbb{R}^d)$. We have shown in \cite{felu06} that if 
  $f\in L^2(\mathbb{R}^d)$ and $g_i\in M^1_v(\mathbb{R}^d)$ then $(V_{g_i}f\cdot\overline{V_{g_i}f(\lambda)})_\lambda$ is absolutely convergent for $i=1,...,n$. Consequently the module frame condition \eqref{ModFrameCond} holds for all $f\in L^2(\mathbb{R}^d)$. Secondly, note that $S_\Lambda$ is an element of $\mathcal{A}^1_v(\Lambda,c)$ and therefore by the spectral invariance of $\mathcal{A}^1_v(\Lambda,c)$ in $C^*(\Lambda,c)$ we obtain that the invertibility on ${_\Lambda}V$ is equivalent to the invertibility of $S_\Lambda$ on $L^2(\mathbb{R}^d)$. 
\end{proof}
If $v$ is a GRS- weight, then by Theorem \ref{thm:MoritaNCWiener} we can choose the $g_1,...,g_n$ in $M^1_v(\mathbb{R}^d)$. In other words there exist standard module frames $\{g_1,...,g_n\}$ for ${_\Lambda}V$ with $g_1,...,g_n\in M^1_v(\mathbb{R}^d)$. Consequently there exist tight multi-window Gabor frames $\mathcal{G}(g_1,...,g_n;\Lambda)$ for $L^2(\mathbb{R}^d)$ with windows $g_1,...,g_n\in M^1_v(\mathbb{R}^d)$. By a theorem of Feichtinger and Gr\"ochenig in \cite{fegr97} a multi-window Gabor frame $\mathcal{G}(g_1,...,g_n;\Lambda)$ for $L^2(\mathbb{R}^d)$ with $g_1,..,g_n\in M^1_v(\mathbb{R}^d)$ is a Banach frame for the class of modulation spaces $M^{p,q}_m(\mathbb{R}^d)$ for a $v$-moderate weight $m$. Note that $m$ is a GRS-weight, too. 
\begin{theorem}[Main result]
Let $\Lambda$ be a lattice in $\mathbb{R}^{2d}$ and $v$ a GRS-weight. Then $M^1_v(\mathbb{R}^d)$ is a finitely generated projective left $\mathcal{A}^1_v(\Lambda,c)$-module. Consequently, there exist $g_1,...,g_n\in M^1_v(\mathbb{R}^d)$ such that $\{g_i:i=1,...,n\}$ is a standard tight $\mathcal{A}^1_v(\Lambda)$-module frame, i.e. 
    \begin{equation}\label{ModFrame}
      f=\sum_{i=1}^n {_\Lambda}\langle f,g_i\rangle\cdot g_i,~~\text{for}~~f\in M^1_v(\mathbb{R}^d).
    \end{equation}
 Consequently, $\mathcal{G}(g_1,...,g_n;\Lambda)$ is a multi-window Gabor frame for the class of modulation spaces $M^{p,q}_m(\mathbb{R}^d)$ for any $v$-moderate weight $m$.   
\end{theorem}
\par
The preceding discussion and our results about the finitely generated projective left modules $\mathscr{S}_v(\mathbb{R}^d)$ over smooth generalized noncommutative tori $\mathcal{A}^\infty_v(\Lambda,c)$ for $v$ a GRS-weight implies the existence of multi-window Gabor frames $\mathcal{G}(g_1,...,g_n;\Lambda)$ for $M^{p,q}_m(\mathbb{R}^d)$ with $g_1,...,g_n\in \mathscr{S}_v(\mathbb{R}^d)$. The result of Feichtinger and Gr\"ochenig also applies to multi-window Gabor frames $\mathcal{G}(g_1,...,g_n;\Lambda)$ for $L^2(\mathbb{R}^d)$ with $g_1,...,g_n\in\mathscr{S}_v(\mathbb{R}^d)$. In the case that $v$ grows like a polynomial we get the existence of multi-window Gabor frames $L^2(\mathbb{R}^d)$ with $g_1,...,g_n\in\mathscr{S}(\mathbb{R}^d)$. 
\par
We summarize these observations about the existence of good multi-window Gabor frames for the class of modulation spaces in the following theorem.
\begin{theorem}
Let $\Lambda$ be a lattice in $\mathbb{R}^{2d}$ and $v$ a GRS-weight. Then there exist $g_1,...,g_n$ in $\mathscr{S}_v(\mathbb{R}^d)$ such that 
$\mathcal{G}(g_1,...,g_n;\Lambda)$ is a multi-window Gabor frame for the class of modulation spaces $M^{p,q}_m(\mathbb{R}^d)$ for any $v$-moderate weight $m$.
\end{theorem}
The link between the projective modules ${_\Lambda}V$ over $C^*(\Lambda,c)$ and multi-window Gabor frames provides one with the possibility to transfer methods and results from Gabor analysis to study properties of ${_\Lambda}V$. In \cite{grluXX} we determine the relation between the number of generators of ${_\Lambda}V$ and the lattice $\Lambda$ and we were able to show that for a lattice $\Lambda$ with $n-1\le\mathrm{vol}(\Lambda)<n$ one needs at least $n$ generators for ${_\Lambda}V$. 
\section{Conclusion}
The most general framework for the present investigation is the time-frequency plane $G\times\widehat G$ for $G$ a locally compact abelian group. All our methods and techniques work for a lattice $\Lambda$ in $G\times\widehat G$, because the twisted group $C^*$-algebra $C^*(\Lambda,c)$ and the subalgebras $\mathcal{A}^1_v(\Lambda,c),\mathcal{A}^\infty_v(\Lambda,c)$ for a GRS-weight $v$ are defined only in terms of time-frequency shifts. Furthermore the definition and properties of modulation spaces and Schwartz-type spaces remain valid in this general setting \cite{fe81-3}. Therefore our main results about projective modules over $C^*(\Lambda,c)$ and the subalgebras $\mathcal{A}^1_v(\Lambda,c),\mathcal{A}^\infty_v(\Lambda,c)$ hold in this very general setting. Finally these observations yield the existence of good multi-window Gabor frames $\mathcal{G}(g_1,...,g_n;\Lambda)$ for $g_1,...,g_n$ in $M^1_v(G)$ or in $\mathscr{S}_v(G)$ for a GRS-weight $v$ on $G\times\widehat G$. We will come back to this topic in forthcoming work.
\section{Acknowledgment}
Large parts of this manuscript were written during several visits at the 
Max Planck Institute for Mathematics at Bonn and we would like to express the deepest gratitude for hospitality and 
excellent working conditions. Furthermore we want to thank Prof. Gr\"ochenig, Dr. D\"orfler and especially Brendan Farell and Rob Martin for helpful remarks on earlier versions of the manuscript. Finally we want to thank an anonymous referee for a very careful proofreading of the manuscript for pointing out the reference \cite{blkuro92}, which led to a considerable improvement of the manuscript and allowed us to extend Theorem 4.2 to the higher-dimensional case. 

\end{document}